\documentclass[pdflatex,sn-mathphys-num]{sn-jnl}

\usepackage{graphicx}%
\usepackage{multirow}%
\usepackage{amsmath,amssymb,amsfonts}%
\usepackage{amsthm}%
\usepackage{mathrsfs}%
\usepackage[title]{appendix}%
\usepackage[dvipsnames]{xcolor}%
\usepackage{textcomp}%
\usepackage{manyfoot}%
\usepackage{booktabs}%
\usepackage{algorithm}%
\usepackage{algorithmicx}%
\usepackage{algpseudocode}%
\usepackage{listings}%

\usepackage{comment}
\usepackage{enumitem}
\usepackage{overpic}

\usepackage{cleveref}

\newcommand{\F}{\mathfrak{F}}

\newcommand{\C}{\mathbb{C}}

\newcommand{\vv}{\mathbf{v}}

\newcommand{\ww}{\mathbf{w}}
\newcommand{\zz}{\mathbf{z}}
\newcommand{\xx}{\mathbf{x}}
\newcommand{\yy}{\mathbf{y}}

\newcommand{\otimesfrac}[2]{%
\frac{#1}{#2}\kern-0.2em\raisebox{-0.4ex}{\footnotesize$\otimes$}%
}

\theoremstyle{thmstyleone}%
\newtheorem{theorem}{Theorem}[section]
\theoremstyle{thmstyletwo}%

\theoremstyle{thmstylethree}%

\usepackage[]{lineno}

\begin{document}

\title[A Contour Method for Multiparameter Eigenvalue Problems]{A Contour Method for Multiparameter Eigenvalue Problems}

\author*[1]{\fnm{Emil} \sur{Graf}}\email{jeg348@cornell.edu}

\author[1]{\fnm{Alex} \sur{Townsend}}\email{townsend@cornell.edu}

\affil*[1]{\orgdiv{Department of Mathematics}, \orgname{Cornell University}, \orgaddress{\street{212 Garden Ave.}, \city{Ithaca}, \postcode{14853}, \state{NY}, \country{USA}}}

\abstract{Multiparameter eigenvalue problems arise in boundary value problems, stability analysis, and delay-differential equations. Despite their importance, existing methods either require solving extremely large global problems or rely on local iterative techniques that only recover a few eigenvalues at a time. In this work we develop the first contour method for analytic multiparameter eigenvalue problems. The key theoretical ingredient is a new residue formula for multivariate matrix-valued analytic functions, extending Beyn's Keldysh-based residue theorem for meromorphic operator functions to several complex variables.
Using this, we derive a multidimensional analogue of Beyn's contour method that computes all the eigenvalues of a multiparameter eigenvalue problem contained in a prescribed region of $\mathbb{C}^d$. The resulting algorithm targets eigenvalues in a region without constructing preposterously large matrices and is embarrassingly parallelizable. Numerical experiments demonstrate that we can now successfully solve large multiparameter problems arising in applications where existing approaches struggle.
}




\maketitle

\section{Introduction}
\label{sec:intro}

Multiparameter eigenvalue problems arise naturally when several spectral parameters interact simultaneously. They appear in the separation of variables for partial differential equations~\cite{willatzen2011bvp,gheorghiu2012mathieu,plestenjak2015mathieu,faierman1991ode}, in stability analysis of dynamical systems~\cite{pons2017aef,pons2018aef}, in delay-differential equations~\cite{jarlebring2009dde}, in the computation of the signed distance between ellipsoids~\cite{iwata2015signeddist}, in the computation of zero-group-velocity points in waveguides~\cite{kiefer2023waves}, when computing properties of leaky waves~\cite{gravenkamp2025leakywaves}, and in the optimal parameter selection for ARMA and LTI models~\cite{demoor2019lti,vermeersch2019arma,demoor2020lti,agudelo2021lti}. Classical examples include multiparameter Sturm--Liouville problems and coupled boundary value problems, which have been studied extensively since the pioneering work of Atkinson~\cite{atkinson1972multieig}.

Despite this long history and the broad range of applications, the numerical solution of multiparameter eigenvalue problems remains challenging. In contrast to the well-developed theory and algorithms for standard (single-parameter) eigenvalue problems, existing approaches typically either construct extremely large linearizations or rely on local iterative methods that recover eigenvalues one at a time. As a result, there is currently no analogue of the powerful contour-based eigensolvers that exist for the computation of eigenvalues for one-parameter problems.

Let $U \subset \C^d$ be an open set and let $\Omega(U,\C^{n_i \times n_i})$ denote the space of analytic functions from $U \to \C^{n_i \times n_i}$. An analytic multiparameter eigenvalue problem (MEP) takes the form
\begin{equation} \label{eq:mepform}
P_i(z_1,\ldots,z_d) \mathbf{v}_i = 0, \quad 1 \leq i \leq d,
\end{equation}
where $P_i(z_1,\ldots,z_d) \in \Omega(U,\C^{n_i \times n_i})$ for integers $n_1,\ldots,n_d$.   The goal is to find points $\mathbf z^*=(z_1^*,\ldots,z_d^*)\in\C^d$ for which there exist nonzero vectors $\mathbf v_i\in\C^{n_i}$ satisfying~\cref{eq:mepform}.  Such a point $\mathbf z^*$ is called an eigenvalue of the MEP,  and the collection $\{\mathbf v_i\}_{i=1}^d$ is called a (right) eigenvector.  Similarly, associated left eigenvectors $\{\mathbf w_i\}_{i=1}^d$ satisfy $\mathbf w_i^H P_i(\mathbf z^*)=0$ for $1\le i\le d$.  As in the scalar case, multiplicities can be defined through the determinant system
$
p(\mathbf z)=\{\textrm{det}(P_i(\mathbf z))\}_{i=1}^d .
$
The algebraic multiplicity of an eigenvalue $\mathbf z^*$ is the multiplicity of $\mathbf z^*$ as a root of this system. Eigenvectors $\{\mathbf v_i^{(1)}\},\ldots,\{\mathbf v_i^{(k)}\}$ are said to be independent if the Kronecker products
$\mathbf v_1^{(j)}\otimes\cdots\otimes\mathbf v_d^{(j)}$ are linearly independent.
The number of independent eigenvectors is the geometric multiplicity.  An eigenvalue is semisimple if these multiplicities coincide, and simple if the algebraic multiplicity is one. Simple eigenvalues are generic, and therefore it is sufficient for a practical method to study only simple eigenvalues.

If the functions $P_i$ are polynomial in $z_1,\ldots,z_d$, we call \eqref{eq:mepform} a polynomial multiparameter eigenvalue problem (PMEP). If, in addition, each $P_i$ is linear in the variables $z_1,\ldots,z_d$, we obtain a linear multiparameter eigenvalue problem (LMEP). Linear multiparameter eigenvalue problems predominate in the literature and are often referred to simply as multiparameter eigenvalue problems.

Solution methods for multiparameter eigenvalue problems are severely limited by the potentially large number of solutions and the correspondingly large matrices that must be constructed. For linear problems, the method of operator determinants~\cite{atkinson1972multieig}, which is the only global algebraic method of which we are aware, requires solving an eigenvalue problem of size $\prod_{i=1}^d n_i$. Homotopy methods also aim to compute all solutions, but encounter similar difficulties with large-scale problems and may be substantially less stable than linear algebra approaches~\cite{plestenjak2000homotopy,plestenjak2001homotopy,dong2016homotopy,rodriguez2021homotopy}.

Because many multiparameter eigenvalue problems are too large to solve globally, a variety of local and subspace methods have been developed, e.g.,~\cite{hochstenbach2005jacobi,hochstenbach2002jacobi,meerbergen2015sylvester,hochstenbach2019subspace}.  These methods can compute a prescribed number of eigenvalues near a given target, or identify extremal eigenvalues such as the smallest or largest ones. Recently, a Newton-type method was proposed that computes eigenvalues indexed by a multiindex~\cite{eisenmann2025newton}.  While local methods can be very efficient, they provide limited control over which eigenvalues are returned. In many applications it is desirable to compute all eigenvalues within a specified region (see \cref{sec:num}), since multiparameter systems typically possess many solutions that are physically irrelevant. To the best of our knowledge, no existing method can certifiably compute all eigenvalues within a region without first computing all solutions in $\C^d$, which is prohibitively expensive.

For polynomial multiparameter eigenvalue problems the situation is even more complicated. Most methods reduce the problem to a linear multiparameter eigenvalue problem and then apply one of the approaches described above; for example, quadratic two-parameter eigenvalue problems can be treated in this way~\cite{bor2010quadeig,hochstenbach2012quadeig}. Recently,~\cite{graf2025pmep} introduced a method capable of globally linearizing and solving any PMEP, but at the cost of solving even larger generalized eigenvalue problems.  For nonpolynomial MEPs the situation is worse: we are not aware of any nonlocal method in the literature (a local method is given in~\cite{plestenjak2016mep}). Consequently, such problems currently cannot be solved with guarantees of computing global or regional solution sets.

A similar tradeoff between large and expensive global methods and faster but target-limited local methods exists for univariate eigenvalue problems. In this setting, an effective middle ground has emerged in the form of contour methods, such as FEAST~\cite{polizzi2009feast} for linear pencils and Beyn's method~\cite{asakura2009beyn,beyn2012beyn,yokota2013beyn} for nonlinear eigenvalue problems. These methods compute all eigenvalues inside a region of the complex plane by evaluating contour integrals along its boundary, and therefore avoid wasting computational resources finding solutions that are not required. We illustrate this in \cref{fig:toyexample}.  An additional advantage is that most of the computational work is embarrassingly parallel. Contour methods have since been extended and refined in many directions, including applications to infinite-dimensional differential eigenvalue problems~\cite{austin2013diffop,horning2020feast,colbrook2025infbeyn} and to eigenvector nonlinearities~\cite{claes2023contour}. See also the surveys~\cite{guttel2017nlevp,imakura2016survey}.

\begin{figure}
\begin{center}
\centering
    \begin{minipage}{0.48\textwidth}
        \begin{overpic}[width=\textwidth, tics=10]{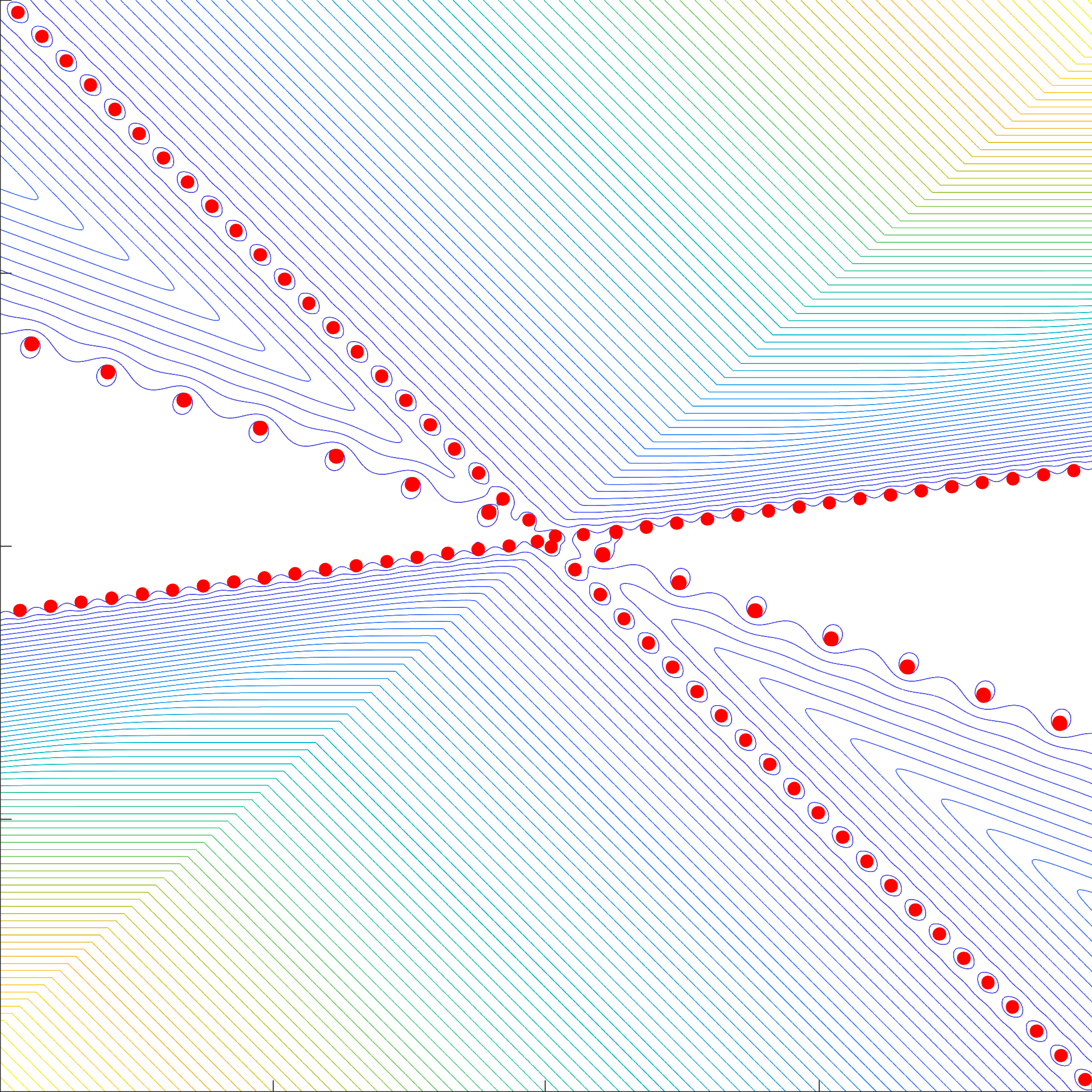}
            \put(44,-9){$\text{Re}(z)$}
            \put(-11,43){\rotatebox{90}{$\text{Im}(z)$}}
             {\footnotesize
            \put(-9,0){$-50$}
            \put(-9,23.5){$-25$}
            \put(-3,49){$0$}
            \put(-5,73){$25$}
            \put(-5,98){$50$}
            \put(-2,-4){$-50$}
            \put(21,-4){$-25$}
            \put(49,-4){$0$}
            \put(73,-4){$25$}
            \put(97,-4){$50$}
            }
        \end{overpic}
    \end{minipage}
    \hspace{.3cm}
\begin{minipage}{0.48\textwidth}
        \begin{overpic}[width=\textwidth, tics=10]{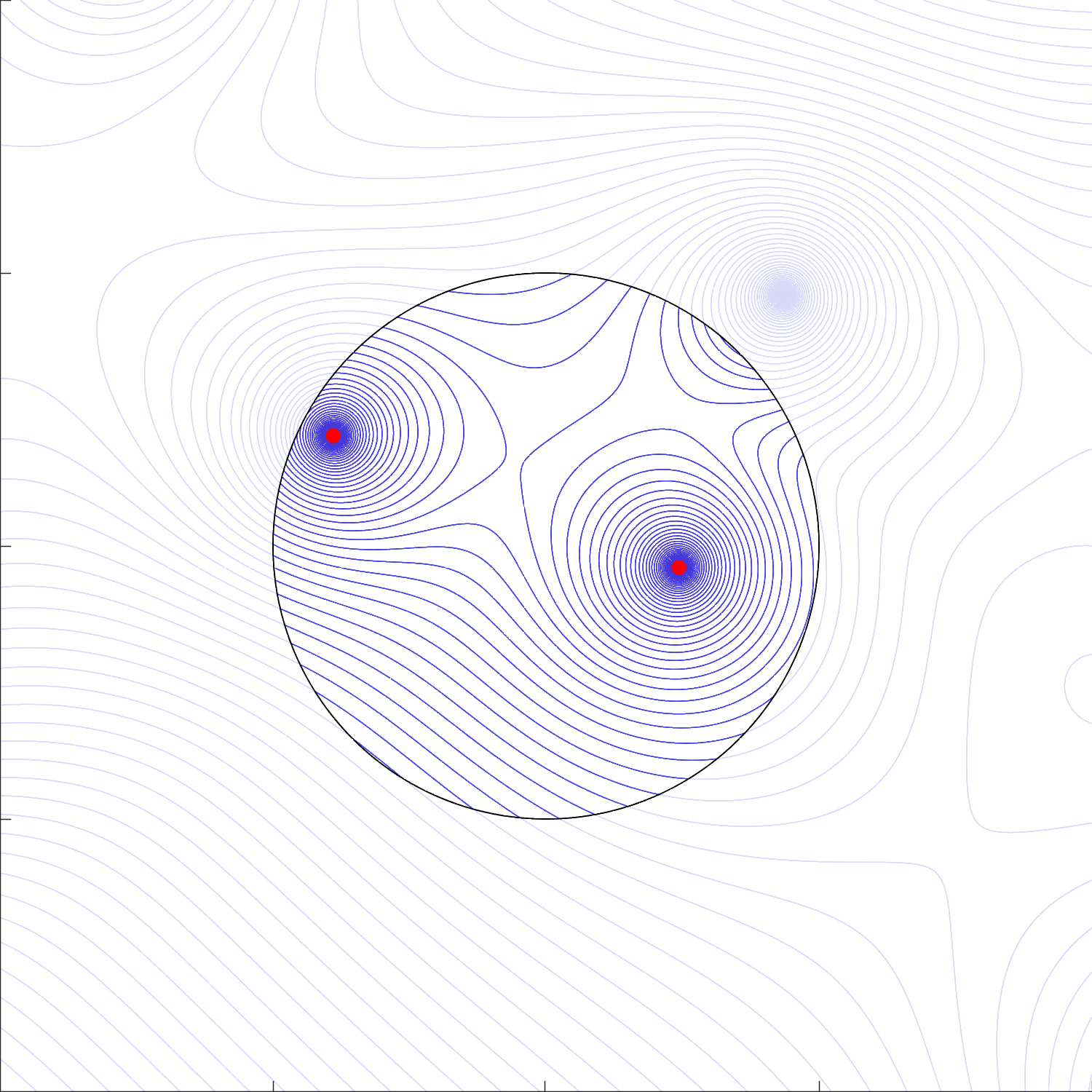}
           \put(44,-9){$\text{Re}(z)$}
             {\footnotesize
            \put(-6.5,0){$-2$}
            \put(-6.5,23.5){$-1$}
            \put(-3,49){$0$}
            \put(-3.5,74){$1$}
            \put(-3.5,98){$2$}
            \put(-2,-4){$-2$}
            \put(21,-4){$-1$}
            \put(49,-4){$0$}
            \put(74,-4){$1$}
            \put(98,-4){$2$}
            }
        \end{overpic}
    \end{minipage}
\end{center}
\vspace{.5cm}
\caption{Left: Contour plot of the smallest singular value of $M(z) = \begin{bmatrix} \sin(e^{\pi i/4}*(z-1/2)) & \sinh(z-1/4) \\ 1/2 & \cos(3z/2-i) \end{bmatrix}$. This analytic matrix function has infinitely many eigenvalues (red dots). Right: Beyn's method applied to $M(z)$ inside the unit circle. The contour solver only ``sees'' the behavior inside the unit circle, so it easily finds the two solutions.}
\label{fig:toyexample}
\end{figure}
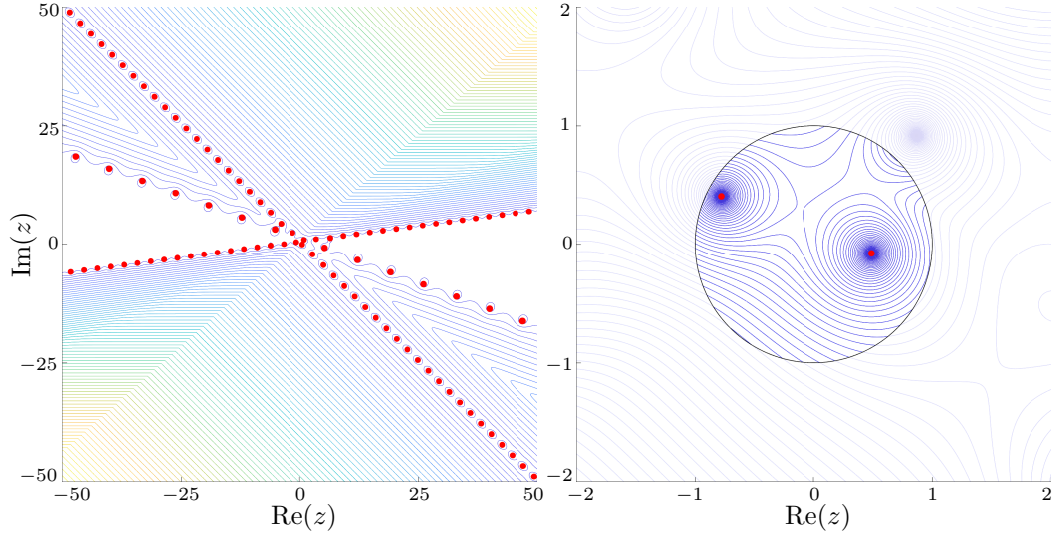

The theoretical foundation of the most general contour methods lies in the work of M.V. Keldysh~\cite{keldysh1951char,keldysh1971complete}, who derived an expansion of a meromorphic operator-valued function of a single variable around a pole (\cref{thm:Keldysh}). This result generalizes the classical theory of residues in complex analysis, and it reduces the justification of contour approaches such as Beyn's method to an application of Cauchy's residue theorem.

For multiparameter eigenvalue problems, however, no contour method currently exists. This leaves a significant gap both in the theory of residues and in the practice of solving eigenvalue problems. The challenge of developing such a method is to produce a multivariate analogue of the Keldysh-based residue formula, i.e., to develop a multidimensional residue formula for matrix functions. The difficulty arises from two sources. First, multidimensional complex analysis lacks a direct analogue of the simple residue theorem of one complex variable. Second, the matrices appearing in multiparameter eigenvalue problems introduce additional complications due to noncommutativity.

One possible approach is to try to obtain a Keldysh-style expansion of a multivariate matrix function in the neighborhood of a singularity. However, singularities of multivariate functions lie along higher dimensional hypersurfaces instead of at points, so such a generalization has more in common with the Weierstrass preparation theorem than with a contour method. Some interesting work has been done in this area~\cite{dencker1993preparation,balicki2026parametric}, but we do not believe that it leads to a contour solver for MEPs.

In multidimensional complex analysis, there are many residue formulas available that generalize Cauchy's residue formula; the challenge then becomes how to navigate them to select results that generalize well to the matrix setting. In~\cite{kravanja1998contour,kravanja2000contour}, the authors develop a contour solver for systems of multivariate scalar analytic functions based on the Bochner--Martinelli integral formula (see~\cite{kytmanov1995bochner,kytmanov2015residue}) and the associated Roos--Yuzhakov logarithmic residue~\cite{roos1974residue,yuzhakov1973residue,alzenberg1983residue}, which might seem to be a reasonable path to generalize. However, careful inspection of the Bochner--Martinelli formula and its proof suggests that formulating such a result is difficult for matrix-valued functions. For this reason, we do not believe that this approach generalizes well to MEPs. Instead, our approach is based on the closely related Grothendieck residue~\cite{griffiths1978ag,soares2002residues,berenstein2012residue} for scalar analytic systems, which we have managed to generalize completely to matrix functions. 

The main result of this paper is a multidimensional residue formula for matrix functions that extends Beyn's Keldysh-based residue formula and the scalar Grothendieck residue. The formula expresses a contour integral constructed from the matrix functions $P_i(z)$ over a $d$-dimensional contour as a sum over eigenvalues and associated eigenvectors (see \cref{thm:simplematrixgrothres,thm:globalrescoho,thm:omegarep,thm:globalres}). This identity leads directly to a contour-based numerical method that computes all eigenvalues inside a prescribed region of $\mathbb{C}^d$ without constructing the exponentially large operator determinants that arise in classical approaches. This solver fills the gap between global and local methods (\cref{sec:alg}). We demonstrate the effectiveness of the method on applications including ARMA models, multiparameter Sturm--Liouville problems, and delay-differential equations (\cref{sec:num}).

\section{Background}
\label{sec:back}
We begin with a brief overview of the analytic tools that underlie our method. In \cref{subsec:keldysh} we review Keldysh's theorem and Beyn's univariate contour method, which together provide the theoretical foundation for contour-based eigensolvers in the single-parameter setting. Our approach may be viewed as a multivariate extension of these ideas.  We also recall the Grothendieck residue for scalar analytic systems, following the treatment in~\cite{griffiths1978ag}, which provides the starting point for our construction.  


\subsection{Keldysh's Theorem and Beyn's Contour Method}
\label{subsec:keldysh}

Univariate contour methods are based on the resolvent of an operator.  For a matrix $A\in\mathbb{C}^{n\times n}$, the resolvent is the meromorphic function $z\mapsto (A-z I)^{-1}$. More generally,  for an analytic matrix function  $A(z)$,  the resolvent is $A(z)^{-1}$.  In both cases the resolvent has poles precisely at the eigenvalues of $A$,  which allows these eigenvalues to be extracted using contour integrals.  A fundamental theoretical result is the theorem of M.\,V.~Keldysh~\cite{keldysh1951char,keldysh1971complete}, which provides a local expansion of the resolvent in a neighborhood of a pole.  A comprehensive modern reference is~\cite{mennicken2003keldysh}. 
We give a brief outline following the presentations in~\cite{beyn2012beyn,guttel2017nlevp}.

Let $U \subset \C$ be an open set. Denote by $\Omega(U,\C^{n \times n})$ the space of holomorphic functions from $U \to \C^{n \times n}$.  
A holomorphic matrix function $A(z)$ has only isolated eigenvalues; that is, for each eigenvalue $\lambda$ there exists a neighborhood $U$ such that $\lambda$ is the only eigenvalue in $U$~\cite[Thm.~1.3.1]{mennicken2003keldysh}. Consequently, the resolvent $A(z)^{-1}$ is meromorphic in $U$ and admits a convergent Laurent expansion about $\lambda$. Keldysh's theorem describes the structure of the singular part of this expansion (see \cref{thm:Keldysh}). When the eigenvalue is simple, the singular term can be written explicitly in terms of the associated left and right eigenvectors, which is sufficient for most practical purposes. A fully general statement appears in~\cite[Thm.~1.6.5]{mennicken2003keldysh}, from which our formulation follows as a specialization.

\begin{theorem}[Keldysh] \label{thm:Keldysh}
	Let $A \in \Omega(U,\C^{n \times n})$ be a holomorphic matrix function and let $C \subset U$ be compact containing finitely many eigenvalues $\lambda_1,\ldots,\lambda_m$. Suppose that all eigenvalues are simple, with right eigenvectors
	$
	\vv_k
	$
	and left eigenvectors
	$
	\ww_k
	$
	for $k = 1,\ldots, m$,
	scaled such that
	\begin{equation} \label{eq:keldyshorth}
		\ww_k^H A'(\lambda_k)\vv_k = 1, \quad k = 1, \ldots, m.
	 \end{equation}	
	Then there exists a neighborhood $C \subset U' \subset U$, and a function $R \in \Omega(U',\C^{n \times n})$ such that, for $z \in U' \setminus \{\lambda_1,\ldots,\lambda_m\}$,
	\begin{equation} \label{eq:simpleglobalkeldysh}
		A(z)^{-1} = \sum_{k=1}^{m}  (z-\lambda_k)^{-1} \vv_{k} \ww_{k}^H + R(z).
\end{equation}
\end{theorem}
Keldysh's theorem shows that the singularities of the resolvent of $A$ in $U$ are completely determined by the eigenstructure of $A$ on $C$. In particular, this implies that the eigenstructure can be recovered by evaluating contour integrals around the boundary of a region containing $C$. The following formula, which follows directly from Cauchy's residue theorem together with \cref{thm:Keldysh}, forms the basis of Beyn's contour solver~\cite[Thm.~2.9]{beyn2012beyn}:
\begin{equation} \label{eq:simplebeyn}
\frac{1}{2\pi i}\int_\Gamma f(z)A(z)^{-1}\,dz
=
\sum_{k=1}^{m} f(\lambda_k)
\vv_{k}\, \ww_{k}^H,
\end{equation}
where $f\in\Omega(U,\mathbb{C})$ and $\Gamma$ is a smooth contour in $U$ enclosing the $m$ eigenvalues $\lambda_1,\ldots,\lambda_m$ of $A$.  Beyn's method computes the eigenvalues inside $\Gamma$ by numerically approximating \cref{eq:simplebeyn} for the choices $f(z)=1$ and $f(z)=z$.  These integrals generate a reduced eigenvalue problem whose size equals the number of eigenvalues inside $\Gamma$, which can yield a substantial computational advantage over global methods that attempt to compute the entire spectrum.

\subsection{The Scalar Grothendieck Residue}
\label{subsec:groth}

Extending contour methods to several complex variables requires a multivariate notion of residue.  In this work, we use the Grothendieck residue for analytic systems~\cite{griffiths1978ag,soares2002residues}, which is closely related to the Roos--Yuzhakov logarithmic residue~\cite{roos1974residue,yuzhakov1973residue,alzenberg1983residue} arising from the Bochner--Martinelli integral formula (see~\cite{kytmanov1995bochner}). 

Unlike the univariate case, zeros of a scalar-valued analytic function in several variables typically occur along hypersurfaces rather than isolated points.  Instead one considers a system of $d$ analytic functions $p_i:\mathbb{C}^d \to \mathbb{C}$,  $i=1,\ldots,d. $ Such systems generically have isolated common roots, making it natural to associate a residue with the system.  Let $\mathbf z^*$ be an isolated common root of $p=\{p_1,\ldots,p_d\}$ and  $U$ be a neighborhood of $\mathbf z^*$ in which $\mathbf z^*$ is the only common root.  Define
\[
\Gamma_\epsilon=\{\mathbf z:\ |p_i(\mathbf z)|=\epsilon,\ 1\le i\le d\},
\]
with orientation determined by $d(\arg(p_1))\wedge\cdots\wedge d(\arg(p_d))\ge 0,$
and assume $\epsilon$ is sufficiently small that $\Gamma_\epsilon\subset U$. For $g\in\Omega(U,\mathbb C)$, the Grothendieck residue of $p$ with respect to $g$ is
\begin{equation}\label{eq:scalarGroth}
\operatorname{Res}_p(g,\mathbf z^*)
=
\left(\frac{1}{2\pi i}\right)^d
\int_{\Gamma_\epsilon}
\frac{g(\mathbf z)\,dz_1\wedge\cdots\wedge dz_d}
{p_1(\mathbf z)\cdots p_d(\mathbf z)} .
\end{equation}
In \cref{sec:theory} we introduce an analogous construction for matrix-valued functions. In both settings it is necessary to relate the residue to the roots of $p$.   
If the root $\mathbf z^*$ is simple, then $\textrm{det}(Jp(\mathbf z^*))\neq 0$, where $Jp$ denotes the Jacobian of the system, and~\cite[Thm.~3.2.2]{soares2002residues} shows that
\begin{equation}\label{eq:simplescalarGrothres}
\operatorname{Res}_p(g,\mathbf z^*)
=
\frac{g(\mathbf z^*)}{\textrm{det}(Jp(\mathbf z^*))}.
\end{equation}
Simple roots are generic, and therefore sufficient for practical purposes. This formula suggests that residues encode the roots of $p$. However,  the definition in~\eqref{eq:scalarGroth} is local and therefore difficult to evaluate without already knowing the root.  Fortunately,  Griffiths and Harris proved that~\cite[p.~654]{griffiths1978ag}
\begin{equation} \label{eq:scalarGrothcoho}
	\operatorname{Res}_p(g,\zz^*) =\! \frac{d!}{(2 \pi i)^d} \!\!\int_{\partial B_{\epsilon}(\zz^*) }\!\!\!\!\!\! \!\!\!\!\frac{g(\zz) \sum_{i=1}^d (-1)^{i-1} \overline{p_i}  d \overline{ p_1} \wedge \!\cdots\! \wedge \widehat{d \overline{p_i}} \wedge \!\cdots  \!\wedge d \overline{p_d} \wedge d \zz}{\|p\|^{2d}},
\end{equation}
where the notation $\widehat{d \overline{p_i}}$ indicates omitting the $i$-th differential, $d \zz = d z_1 \wedge \!\cdots\! \wedge d z_d$, and the integral is over the boundary of an $\epsilon$-ball around $\zz^*$. The result is proved by viewing the residue in a different cohomology theory, by means of the Dolbeault isomorphism. We will prove a similar result for analytic matrix functions, using generalizations of the same tools (see~\cref{thm:globalres}).

Once we have the form of the residue in~\cref{eq:scalarGrothcoho}, it is straightforward to globalize it. If $\Gamma$ is a contour enclosing roots $\zz_1,\ldots,\zz_{m}$, then we obtain~\cite[p. 656]{griffiths1978ag}
\begin{equation} \label{eq:globalscalarGroth}
	\sum_{i=1}^{m} \operatorname{Res}_p(g,\zz_i) = \frac{d!}{(2 \pi i)^d} \int_{\Gamma} \frac{g(\zz) \sum_{i=1}^d (-1)^{i-1} \overline{p_i}  d \overline{ p_1} \wedge \!\cdots\! \wedge \widehat{d \overline{p_i}} \wedge \!\cdots  \!\wedge d \overline{p_d} \wedge d \zz}{\|p\|^{2d}}.
\end{equation}

From~\cref{eq:globalscalarGroth} it is natural to consider a contour method for systems of scalar analytic functions that is analogous to Beyn's method. The methods developed in~\cite{kravanja1998contour,kravanja2000contour} are based on an equivalent residue formula involving the Bochner--Martinelli kernel, which provides a multivariate generalization of Cauchy's integral formula~\cite[Thm.~1.4]{kytmanov1995bochner}. Although our derivation follows the framework of the Grothendieck residue, the resulting algorithm is closely related to, and may be viewed as a generalization of, the contour methods proposed in~\cite{kravanja1998contour,kravanja2000contour}.

\subsection{Dolbeault's Theorem}
\label{subsec:prelim}


The multidimensional residue formula that underlies our contour method is most naturally expressed using the language of differential forms and cohomology.  In the scalar case, Griffiths and Harris~\cite{griffiths1978ag} developed the theory by viewing the Grothendieck residue in two different cohomology theories, and extracting explicit differential forms to represent the cohomology classes.  The derivation of our contour formula in \Cref{sec:theory} requires us to transform an integral of a holomorphic form over a small polydisc into an integral of a $(d,d-1)$-form over the boundary of a region in $\C^d$.  Carrying out this transformation requires Dolbeault's theorem with coefficients in a vector bundle, which provides a connection between holomorphic forms and $(p,q)$-forms. 

Let $X$ be a complex manifold of complex dimension $d$ and let
$E\to X$ be a holomorphic vector bundle.
We denote by $A^{p,q}(X,E)$ the sheaf of smooth $(p,q)$-forms on $X$
with values in $E$. In local holomorphic coordinates
$z_1,\ldots,z_d$, such a form can be written as
\[
\sum_{I,J}
\phi_{I,J}(z)
\,dz_{i_1}\wedge\cdots\wedge dz_{i_p}
\wedge
d\overline z_{j_1}\wedge\cdots\wedge d\overline z_{j_q},
\]
where the coefficients $\phi_{I,J}(z)$ are smooth sections of $E$.  The Dolbeault operator
\[
\overline{\partial}:A^{p,q}(X,E)\rightarrow A^{p,q+1}(X,E)
\]
acts by differentiating the coefficients with respect to the
$\overline z$ variables~\cite[p.~87]{kodaira2005manifolds}.  We denote by
\[
Z^{p,q}(X,E)=\ker(\overline{\partial}:A^{p,q}(X,E)\to A^{p,q+1}(X,E))
\]
the subsheaf of $\overline{\partial}$-closed $(p,q)$-forms.  Holomorphic $p$-forms with values in $E$ form a sheaf, denoted $\Omega^p(E)$.  The associated sheaf cohomology groups $H^q(X,\Omega^p(E))$, as defined in~\cite[p.~115-123]{kodaira2005manifolds}, arise naturally in the theory of multivariate residues. 

Dolbeault's theorem is an identification between sheaf cohomology groups and Dolbeault cohomology groups; with coefficients in a vector bundle of differential forms, it allows us to replace $d$-forms with $(d,d-1)$-forms. While the result below is classical (see~\cite[Thm.~3.20]{wells2008manifolds}),  we give a proof that closely follows the scalar argument
presented in~\cite[p.~45]{griffiths1978ag} because the explicit form of the isomorphism is used later when we
construct representatives of the relevant cohomology classes.

\begin{theorem}[Dolbeault]\label{thm:dolbeault}
Let $X$ be a complex manifold and let $E$ be a holomorphic vector bundle on
$X$. Then
\begin{equation}
H^q(X,\Omega^p(E)) \cong H^{p,q}(X,E).
\end{equation}
where
$
H^{p,q}(X,E)
=
\frac{Z^{p,q}(X,E)}
{\overline{\partial}A^{p,q-1}(X,E)}
$
is the Dolbeault cohomology group.
\end{theorem}
\begin{proof}
Let $Z^{p,s}(E)$ denote the subsheaf of $\overline{\partial}$-closed
$(p,s)$-forms with values in $E$, that is, the kernel of 
$\overline{\partial}$.  For $s \geq 0$, there is a short exact sequence of
sheaves
\[
0
\longrightarrow
Z^{p,s}(E)
\longrightarrow
A^{p,s}(E)
\xrightarrow{\overline{\partial}}
Z^{p,s+1}(E)
\longrightarrow
0.
\]
Taking sheaf cohomology yields the associated long exact sequence of sheaf cohomology groups $H^p$ (see~\cite[Thm.~3.7]{kodaira2005manifolds})
\begin{align*}
0 \to H^0(X,Z^{p,s}) &\to H^0(X,A^{p,s}) \to H^0(X,Z^{p,s+1}) \to \\
H^1(X,Z^{p,s}) &\to H^1(X,A^{p,s}) \to H^1(X,Z^{p,s+1}) \to \\
&\vdots \\
H^r(X,Z^{p,s}) &\to H^r(X,A^{p,s}) \to H^r(X,Z^{p,s+1}) \to \cdots .
\end{align*}
By the argument of~\cite[p.~42]{griffiths1978ag}, as $A^{p,s}$ is a fine sheaf, $H^r(X,A^{p,s})=0$ for $r>0$,\footnote{$A^{p,s}$ is the vector bundle-valued version of $\mathscr{P}^{p,s}$ from~\cite[p.~36]{griffiths1978ag}. See also~\cite[p.~134-135]{kodaira2005manifolds}.} so the connecting homomorphism in the
long exact sequence induces isomorphisms
$
i_{r,s} : H^r(X,Z^{p,s})
\longrightarrow
H^{r-1}(X,Z^{p,s+1})
$
for $r > 1$  and $s \geq 0$. 
Iterating these maps produces the sequence of isomorphisms
\begin{align*}
H^q(X,\Omega^p(E)) &\cong H^{q-1}(X,Z^{p,1}) \cong H^{q-2}(X,Z^{p,2}) \cong \cdots \cong H^1(X,Z^{p,q-1}) \\&\cong \frac{H^0(X,Z^{p,q})}{\overline{\partial}H^0(X,A^{p,q-1})} \cong H^{p,q}(X,E),
\end{align*}
which establishes the theorem.
\end{proof}

Dolbeault's theorem allows us to pass from holomorphic $d$-forms to smooth $(d,d-1)$-forms, which we can integrate over the boundary of $d$-dimensional complex regions.  In the
derivation of our multidimensional residue formula we will pass a cohomology class through the
maps $i_p$ to obtain an explicit representative of the
corresponding Dolbeault cohomology class.

\section{A Multiparameter Residue for Analytic Matrix Functions}
\label{sec:theory}
In this section we develop our multiparameter residue formula for analytic matrix functions.  Our derivation follows the same general strategy as used in the derivation of the scalar Grothendieck residue in~\cite[p. 647-656]{griffiths1978ag}. We consider an analytic multiparameter eigenvalue problem $P$ as in \cref{eq:mepform} with analytic matrix functions $P_i \in \Omega(U,\C^{n_i \times n_i})$. 
Our goal is to derive an identity of the form
\[
\int_{\partial C} \Phi(\mathbf z)
=
\sum_{\mathbf z^* \in C}
\mathcal{R}(\mathbf z^*),
\]
where $\Phi$ is a differential form constructed from the matrix functions $P_i$ and $\mathcal{R}(\mathbf z^*)$ encodes the eigenstructure of the multiparameter eigenvalue problem at each eigenvalue $\mathbf z^* \in C$. Such an identity allows eigenvalues inside a region $C \subset \mathbb C^d$ to be recovered from contour integrals over $\partial C$.
Our derivation proceeds in four main steps:
\begin{enumerate}[leftmargin=*]
\item We define a local residue associated with an isolated eigenvalue (see~\cref{eq:matrixgrothres}).
\item We compute this residue explicitly when the eigenvalue is simple (see~\cref{thm:simplematrixgrothres}).
\item Using Dolbeault's theorem, we convert the local residue integral into an
integral over the boundary of a standard region in $\C^d$ (see~\cref{thm:globalrescoho,thm:omegarep}).
\item Finally we sum these contributions to obtain a global contour formula (see~\cref{thm:globalres}).
\end{enumerate}
The resulting identity is the multidimensional analogue of the residue formula underlying Beyn's contour method and gives us a contour-based eigensolver for MEPs (see~\cref{sec:alg}).

A fundamental obstacle in extending residue formulas to matrix-valued functions is the noncommutativity of matrix multiplication.  In the scalar theory the order of factors in expressions such as $p_1(z)\cdots p_d(z)$ is irrelevant, whereas for matrices the order of multiplication matters.  To circumvent this difficulty we embed the multiparameter eigenvalue problem into a larger commuting system using Kronecker products.  This idea appears in the classical operator determinant method~\cite{atkinson1972multieig} and in recent work on polynomial MEPs~\cite{graf2025pmep}.   Specifically,  we define
\begin{equation}
Q_i =
I_{n_1}\otimes\cdots\otimes I_{n_{i-1}}
\otimes P_i
\otimes I_{n_{i+1}}\otimes\cdots\otimes I_{n_d},
\qquad
1\le i\le d.
\end{equation}
The matrices $Q_1,\ldots,Q_d$ have a common eigenvector given by 
$
\vv_1\otimes\cdots\otimes\vv_d,
$
and commute with one another.  This allows us to apply residue constructions that mimic the scalar case.  

Importantly, these large matrices are introduced only for the theoretical derivation of the residue formula; the numerical algorithm developed later does not construct them explicitly.

\subsection{A Local Residue Associated With an Isolated Eigenvalue}
\label{subsec:locres}

We begin by defining a local residue associated with a single isolated eigenvalue of the MEP.   Suppose that
$\zz^* = (z_1^*,\ldots,z_d^*)$ is an eigenvalue of $P$ and that there exists a neighborhood $U \subset \C^d$ of $\zz^*$ containing no other eigenvalues of $P$. For the purpose of defining a local residue, we restrict attention to the problem $P$ on the neighborhood $U$.

Let $g \in \Omega(U,\C)$ be a holomorphic scalar function.   We define the matrix-valued differential form
\begin{equation} \label{eq:matrixomega}
\omega(\zz) = g(\zz)(Q_1(\zz) \cdots Q_d(\zz))^{-1} d \zz =  \frac{g(\zz)}{Q_1 \cdots Q_d} d \mathbf{z},
\end{equation}
where $d \zz = d z_1 \wedge \cdots \wedge d z_d$.  Equivalently,  we can write 
$$\omega(\zz) = g(\zz)(P_1(\zz) \otimes \cdots \otimes P_d(\zz))^{-1} d \zz.
$$
To define the integration contour, let 
\begin{equation} \label{eq:matrixgamma}
\Gamma = \{ \zz \in U : \lvert \textrm{det}(P_i) \rvert = \epsilon , 1 \leq i \leq d\},
\end{equation}
for a sufficiently small constant $\epsilon$,  chosen so that $\Gamma$ is contained in $U$.  We define the matrix Grothendieck residue of $P$ with respect to $g$ at $\zz^*$ by 
\begin{equation} \label{eq:matrixgrothres}
\operatorname{Res}_P(g,\zz^*) = \left( \frac{1}{2 \pi i } \right)^d \int_{\Gamma} \omega.
\end{equation}

Our first result shows that this residue encodes the eigenstructure of the MEP in a neighborhood of $\zz^*$.  Throughout this section we assume that $\zz^*$ is a simple eigenvalue, meaning that $\zz^*$ is a simple root of the scalar analytic system $\{ \textrm{det}(P_i) \}_{i=1}^d$. 
\begin{theorem} \label{thm:simplematrixgrothres}
Suppose that $\zz^* = (z_1^*,\ldots,z_d^*)$ is a simple eigenvalue of the MEP $P = \{P_i,1 \leq i \leq d\}$ in \cref{eq:mepform}, with right eigenvectors $\vv_1,\ldots,\vv_d$ and left eigenvectors $\ww_1,\ldots,\ww_d$. Furthermore,  assume the eigenvectors are scaled so that
\begin{equation} \label{eq:evscale}
\ww_i^H \frac{\partial P_i}{\partial \textrm{det}(P_i)}(\zz^*) \vv_i = 1, \quad 1 \leq i \leq d.\footnote{For readers familiar with Keldysh's theorem, the scale may seem odd.  This altered formula comes from the fact that we can only apply Keldysh's univariate theorem after a suitable change of coordinates, so the scaling we obtain depends on $\frac{\partial P_i}{\partial \textrm{det}(P_i)}$ rather than on $\frac{\partial P_i}{\partial x_j}$ or some other function. In practice, the scale of the eigenvectors is irrelevant to our method.}
\end{equation}
Let $p_i = \textrm{det}(P_i)$ for $1 \leq i \leq d$ and denote by $Jp$ the Jacobian matrix of the system $p = \{p_i\}$. Then, for sufficiently small $\epsilon$, 
\begin{equation} \label{eq:simplematrixgrothres}
	\operatorname{Res}_P(g,\zz^*) = \frac{g(\zz^*)}{\textrm{det}(Jp(\zz^*))} (\vv_1 \otimes \cdots \otimes \vv_d) (\ww_1 \otimes \cdots \otimes \ww_d)^H .
\end{equation}
\end{theorem}
\begin{proof}
Because the eigenvalue $\zz^*$ is simple,  the functions $x_i = \textrm{det}(P_i)$, $1\leq i\leq d$ define a system of local coordinates in a neighborhood of $\zz^*$. Using these coordinates we rewrite the residue integral to obtain 
	\begin{align*}
	\operatorname{Res}_P(g,\zz^*)  &=  \left( \frac{1}{2 \pi i } \right)^d \int_{\Gamma} \frac{g(\zz)}{P_1 \otimes \cdots \otimes P_d} d \mathbf{z} \\
	&= \left( \frac{1}{2 \pi i } \right)^d \int_{\Gamma} \frac{g^*(\xx)}{P_1^*(\xx) \otimes  \cdots \otimes P_d^*(\xx)} \frac{d \xx}{\textrm{det}(Jp^*(\xx))},
	\end{align*}
where $g^*(\xx)$ denotes $g(\zz(\xx))$, the transformation of $g$ in the new coordinates, and similarly $Jp^*(\xx) = Jp(\zz(\xx))$. We let $P_i^* = P_i^*(\xx) = P_i(\zz(\xx))$ denote the same transformation of coordinates.  In these coordinates the contour decomposes as $\Gamma = \Gamma_1 \times \cdots \times \Gamma_d$, where $\Gamma_i = \{ x_i : |x_i| = \epsilon \}$, so the integral may be evaluated sequentially, one variable at a time.  We first integrate over $\Gamma_1$ with respect to $x_1$,  holding the other variables fixed. The function $\frac{g^*(\xx)}{P_2^*(\xx) \otimes \cdots \otimes P_d^*(\xx) \textrm{det}(Jp^*(\xx))}$ is a holomorphic function in $x_1$,  while $P_1^*(\xx)^{-1} = (P_1^*(x_1))^{-1}$ is a meromorphic function of $x_1$. Therefore, Keldysh's theorem applies to $P_1^*(x_1)$. 

Choosing the contour sufficiently small, for fixed $x_2,\ldots,x_d$,  the matrix function $P_1^*$ has a single, simple eigenvalue inside $\Gamma_1$.  Let $\lambda_1(x_2,\ldots,x_d),\vv_1(x_2,\ldots,x_d),\ww_1(x_2,\ldots,x_d)$ denote this eigenvalue, and the associated right and left eigenvectors.  Then \cref{eq:simpleglobalkeldysh} and \cref{eq:simplebeyn} give us
\begin{align*}
	\operatorname{Res}_P(g,\zz^*)  \!\! &= \!\! \left( \frac{1}{2 \pi i } \right)^d \int_{\Gamma_d}  \cdots \int_{\Gamma_1} \frac{g^*(\xx)}{P_1^*(\xx) \otimes \cdots \otimes P_d^*(\xx)} \frac{d \xx}{\textrm{det}(Jp^*(\xx))} \\
	&= \!\! \left( \frac{1}{2 \pi i } \right)^{d-1} \!\!\!\!\! \int_{\Gamma_d} \!\!\! \cdots \!\! \int_{\Gamma_2} \!\!\! \otimesfrac{g^*(\lambda_1,x_2,\ldots,x_d) \vv_1(x_2,\ldots,x_d)\ww_1(x_2,\ldots,x_d)^H d x_2 \cdots dx_d}{P_2^*(\xx) \otimes \cdots  \otimes P_d^*(\xx) \textrm{det}(Jp^*(\lambda_1,x_2,\ldots,x_d))},\\
\end{align*}
where the fraction with the Kronecker product symbol $\otimesfrac{\mathcal{N}}{\mathcal{D}}$ denotes $\mathcal{N} \otimes \mathcal{D}^{-1}$. Repeating this argument sequentially for the remaining variables produces the eigenvector factors one at a time.  At the final step we integrate with respect to $x_d$ to obtain
\begin{align*}
	\operatorname{Res}_P(g,\zz^*) 
	&= \left( \frac{1}{2 \pi i } \right) \int_{\Gamma_d} \otimesfrac{g^*(\lambda_1,\ldots,\lambda_{d-1},x_d) (\vv_1\otimes \cdots \otimes \vv_{d-1})(\ww_1\otimes \cdots \otimes \ww_{d-1})^H dx_d}{P_d^*(x_d) \textrm{det}(Jp^*(\lambda_1,\ldots,\lambda_{d-1},x_d))},
\end{align*}
where $\vv_i(x_d),\ww_i(x_d),\lambda_i(x_d)$ are functions of $x_d$ that give the eigenvalue-eigenvector triples for $P_i$ as a function of $x_d$. The integrand is singular only at $x_d=0$, where $\vv_i(x_d),\ww_i(x_d),\lambda_i(x_d)$ all collapse to $\vv_i,\ww_i,0$.  Evaluating the final residue gives
\begin{align*}
	\operatorname{Res}_P(g,\zz^*) 
&= \frac{g^*(0)}{\textrm{det}(Jp^*(0))} (\vv_1 \otimes \cdots \otimes \vv_d) (\ww_1 \otimes \cdots \otimes \ww_d)^H
\end{align*}
Finally, the coordinate transformation satisfies
$\xx =0 \iff \zz =\zz^*$,  so rewriting the expression in the original
coordinates yields~\cref{eq:simplematrixgrothres}.
\end{proof}

Simple eigenvalues are generic, and numerical algorithms generally cannot distinguish between a multiple eigenvalue and several tightly clustered simple eigenvalues.  Consequently,  for algorithmic purposes~\cref{thm:simplematrixgrothres} is sufficient. This mirrors the practical implementation of univariate contour solvers, where the simple-eigenvalue formulas in~\cref{thm:Keldysh,eq:simplebeyn} are typically used.  Observe that \cref{thm:simplematrixgrothres} already closely resembles a direct  multiparameter analogue of \cref{eq:simplebeyn}.  It becomes a strict generalization once the contour $\Gamma$ does not need to depend on the MEP $P$. 

\subsection{The Global Residue}
\label{subsec:globalres}

While the local residue provides the key building block of the theory, it is not directly suitable for numerical computation. The contour $\Gamma$ defined in \cref{eq:matrixgamma} depends explicitly on the matrix functions $P_i$,  and is constructed only in a small neighborhood of a single eigenvalue.  To obtain a practical contour method we therefore seek a global formulation analogous to the integral identity in~\cref{eq:simplebeyn}. The development of this section follows the classical 
construction of the Grothendieck residue presented in~\cite[Ch.~5]{griffiths1978ag}.

Our goal is to replace the special contour $\Gamma$ by an integral over the boundary of a standard complex region. This is achieved by passing the differential form $\omega$ through the Dolbeault isomorphism (see~\cref{thm:dolbeault}).  Let $U$ be a neighborhood of an eigenvalue $\zz^* \in \C^d$ containing no other eigenvalues of $P$, and suppose that $U$ is relatively compact with smooth boundary.  Let $N = \prod_{i=1}^d n_i$ denote the size of the matrices $Q_i$,  so that $Q_i \in \Omega(U,\C^{N \times N})$.  For each $1 \leq i \leq d$ define the hypersurfaces $D_i = \{ \textrm{det}(P_i) = 0 \} \cap U$, so that by construction the intersection of these hypersurfaces contains the eigenvalue,  i.e.,  $D_1\cap\cdots\cap D_d = \{\zz^*\}$.   We introduce the punctured domain $U^* = U \setminus (D_1 \cap \cdots \cap D_d) = U \setminus \{\zz^*\}$, and for each $i$,  set $U_i = U \setminus D_i$, so that $\underline{U} = \{U_i\}_{i=1}^d$ is an open cover of $U^*$. 

We use the notation $C^p(\underline{U},\F)$ and $H^p(\underline{U},\F)$  to denote the sheaf cohomology $p$-cochains and the associated cohomology groups for a sheaf $\F$ with respect to the cover $\underline{U}$, following~\cite[p.~115]{kodaira2005manifolds}. For our sheaves of differential forms, an element of $C^p$, or a representative of a class in $H^p$, is a set of differential forms, with one differential form defined on each possible intersection of $p+1$ sets in the cover $\underline{U}$. We also use the sheaf cohomology coboundary operator $\delta:C^{p}(\underline{U},\mathcal{F}) \to C^{p+1}(\underline{U},\mathcal{F})$, defined by~\cite[p. 115--116]{kodaira2005manifolds} as
\begin{equation} \label{eq:delta}
(\delta \sigma)_{i_0,\ldots,i_{p+1}} = \sum_{j=0}^{p+1} (-1)^j \sigma_{i_0 \cdots \hat{i_j} \cdots i_{p+1}} \vert_{U_{i_0} \cap \cdots \cap U_{i_{p+1}}},
\end{equation}
where $i_0 \cdots \hat{i_j} \cdots i_{p+1}$ means to omit $i_j$ from the index.\footnote{The definition in~\cite[p.~38]{griffiths1978ag} is missing the $\hat{}$ over $\hat{i_j}$, so we reference the correct definition in~\cite[p. 115--116]{kodaira2005manifolds} instead.} We can view the differential form $\omega = \frac{g(\zz)}{Q_1 \cdots Q_d} d \zz \in C^{d-1}(\underline{U},\Omega^d(U^*,\C^{N \times N}))$, as elements of $C^{d-1}(\underline{U},\Omega^d(U^*,\C^{N \times N}))$ are holomorphic $d$-forms defined on the neighborhood $\cap_{i=1}^d U_i$. Furthermore, we have $\delta \omega = 0$ trivially because $C^d(\underline{U},\Omega^d(U^*,\C^{N \times N})) = 0$, as the cover $\underline{U}$ only has $d$ elements. Therefore there is a representative of $\omega$ in $H^{d-1}(\underline{U},\Omega^d(U^*,\C^{N \times N}))$, which is the equivalence class of $\omega$ in $H^{d-1}(\underline{U},\Omega^d(U^*,\C^{N \times N}))$, which we also denote by $\omega$.

The next step is to pass from sheaf cohomology to Dolbeault cohomology. In order to obtain concrete representatives we work with cohomology defined with respect to the cover $\underline U$.  By Cartan's theorem B~\cite[p. ~243]{gunning2022analytic},  if every intersection $U_{i_1} \cap \cdots \cap U_{i_{\ell}}$ is a Stein manifold, then, for every such intersection, $H^p(U_{i_1} \cap \cdots \cap U_{i_{\ell}},\Omega^q(U^*,\C^{N \times N})) = 0$ for $p > 0$, so Leray's theorem~\cite[p.~40]{griffiths1978ag} implies that $H^{p}(U^*,\Omega^q(U^*,\C^{N \times N})) = H^{p}(\underline{U},\Omega^q(U^*,\C^{N \times N}))$; the same statement holds for the sheaves $A^{p,q}$ and $Z^{p,q}$ (these sheaves are sheaves of germs of sections of analytic vector bundles,\footnote{For example, $A^{p,q}$ consists of sections of the vector bundle whose points are $(\zz,\nu(\zz))$, for $\zz \in \C^d$ and $\nu \in A^{p,q}$, with projection onto $\C^d$. A local section of this vector bundle, which is a right inverse of the projection map, is just an element of $A^{p,q}$ restricted to some open set.} and therefore coherent sheaves by~\cite[p. 167]{hormander1973scv}). Since each intersection of the open sets $U_i$ is a domain of holomorphy and therefore Stein~\cite[p.~109]{hormander1973scv},  the hypotheses of Leray's theorem are satisfied.  Thus we may work entirely with cohomology defined with respect to the cover $\underline U$.
If we pass $\left( \frac{1}{2 \pi i} \right)^d \omega$ through the Dolbeault isomorphism $H^{d-1}(U^*,\Omega^d(U^*,\C^{N \times N})) \xrightarrow{\sim} H^{d,d-1}(U^*,\Omega(U^*,\C^{N \times N}))$ in \cref{thm:dolbeault},  we obtain a representative $\eta_{\omega} \in H^{d,d-1}(U^*,\Omega(U^*,\C^{N \times N}))$, the image of $\left( \frac{1}{2 \pi i} \right)^d \omega$ under the Dolbeault isomorphism. 
We can naturally integrate elements of $H^{d,d-1}$ on the boundary of arbitrary $d$-dimensional complex regions; in practice $\eta_{\omega}$ is what we can plug into a practical integral formula for our solver. Crucially, while $\eta_{\omega}$ is abstractly a cohomology class, we will be able to obtain a concrete representative in \cref{thm:omegarep}. First, we prove that integrating $\eta_{\omega}$ on the boundary of $U$ gives the local residue of an eigenvalue $\zz^*$,  i.e., the same result as integrating $\omega$ on $\Gamma$.

\begin{theorem} \label{thm:globalrescoho}
	Suppose that $\zz^* = (z_1^*,\ldots,z_d^*)$ is a simple eigenvalue of the MEP $P = \{P_i,1 \leq i \leq d\}$ in~\cref{eq:mepform}.  Let $\omega$ and $\eta_{\omega}$ be defined as in the previous discussion,  and let $\Gamma$ be the contour defined in~\cref{eq:matrixgamma}. Then
	\begin{equation} \label{eq:globalrescoho}
	\left( \frac{1}{2 \pi i} \right)^d \int_\Gamma \omega = \int_{\partial U} \eta_{\omega}
	\end{equation}
\end{theorem}
\begin{proof}
The argument follows the same construction as the scalar case in~\cite[p. 651--653]{griffiths1978ag}.  Since $\delta \omega = 0$, the form $\omega_{d-1} = \left( \frac{1}{2 \pi i} \right)^d \omega$ defines a class in $H^{d-1}(\underline{U},\Omega^d(U^*,\C^{N \times N})) = H^{d-1}(U^*,\Omega^d(U^*,\C^{N \times N}))$. We know that $\eta_\omega$ is obtained from $\omega$ through the sequence of isomorphisms in the proof of \cref{thm:dolbeault}.  We track the form $\omega$ through this sequence of isomorphisms. 

Let $U_I = \cap_{i \in I} U_i$ for $I \subset \{1,\ldots,d\}$. For $p = d-1,\ldots,0$, denote by $\omega_p = \{\omega_{p,I} \in Z^{d,d-p-1}(U_{I})\}_{|I| = p+1} \in H^{p}(\underline{U},Z^{d,d-p-1})$ the image of $\omega_{d-1}$ partway through the chain of isomorphisms used in the proof of \cref{thm:dolbeault}. From the construction of the long exact cohomology sequence~\cite[p.~126--127]{kodaira2005manifolds},  and because we are in a setting where the cohomology groups with respect to $\underline{U}$ are equivalent to the sheaf cohomology groups, there exist forms $\xi_p = \{\xi_{p,I} \in A^{d,d-p-1}(U_{I})\}_{|I| = p} \in C^{p-1}(\underline{U},A^{d,d-p-1})$ such that $\delta \xi_p = \omega_p$ and $\overline{\partial} \xi_p = \omega_{p-1}$,  where $\overline{\partial}$ is the Dolbeault operator, and $\delta:C^{p-1}(\underline{U},A^{d,d-p-1}) \to C^{p}(\underline{U},A^{d,d-p-1})$ is the sheaf cohomology coboundary operator in \cref{eq:delta}.\footnote{The occurrence of $\delta$ and $\overline{\partial}$ in this manner is explained precisely by the construction of the long exact sequence in sheaf cohomology in~\cite[p.~126--127]{kodaira2005manifolds} or~\cite[p.~40]{griffiths1978ag}. This is analogous to the scalar construction in~\cite[p. 651--653]{griffiths1978ag}.}

To relate these forms to contour integrals we introduce the chains
$$
\Gamma_I = \{ \mathbf{z} : |\textrm{det}(P_i(\mathbf{z}))| = \epsilon \text{ for } i \in I, |\textrm{det}(P_j(\mathbf{z}))| \leq \epsilon \text{ for } j \notin I \},
$$
where the orientation is chosen so that
$$
\left(
\bigwedge_{j \notin I}
\frac{\sqrt{-1}}{2}\, dp_j \wedge \overline{dp_j}
\right)
 \wedge d(\arg p_{i_1}) \wedge \cdots \wedge d(\arg p_{i_p})
> 0,\footnote{The textbook~\cite[p.~652]{griffiths1978ag} gives the orientation with the $2$-forms after the $1$-forms, so that $(j,I \cup \{j\})$ in the boundary is the index counting from the back. This disagrees with the definition of $\delta$, where the alternating sign convention counts from the front, so we have modified so that the orientation on the boundary lines up with the alternating sign from $\delta$. The result is valid either way, as the difference is only in a possible global sign.}
$$
with $p_i = \textrm{det} (P_i)$. This corresponds to the standard counterclockwise orientation on the boundary circles for $i\in I$ and the standard orientation of the
interior discs for $j\notin I$.  The boundary of $\Gamma_I$ is
$$
\partial \Gamma_I = \sum_{j \notin I} (-1)^{(j,I \cup \{j\})} \Gamma_{I \cup \{j\}},
$$
where $(j,I \cup \{j\})$ denotes the position of $j$ in the ordered index set $I \cup \{j\}$, with the first element having position $0$. The sign on each term comes from the number of $1$-forms that must be ``stepped over'' for $d(\arg p_j)$ to reach the correct position in the sequence of $1$-forms in $\Gamma_{I \cup \{j\}}$.   Because the exterior derivative satisfies $d = \overline{\partial}$ on $(d,p)$-forms for all $p$,  Stokes' theorem yields
\begin{align*}
    \sum_{|I| = p} \int_{\Gamma_I} \omega_{p-1,I} &= \sum_{|I| = p} \int_{\Gamma_I} d \xi_{p,I}\\
    &= \sum_{|I| = p} \int_{\partial \Gamma_I} \xi_{p,I}\\
    &= \sum_{|I| = p} \left( \sum_{j \notin I} \int_{\Gamma_{I \cup \{j\}}} (-1)^{(j,I \cup \{j\})} \xi_{p,I} \right)\\
    &= \sum_{|J| = p+1} \left(  \int_{\Gamma_{J}} \sum_{j \in J} (-1)^{(j,J)} \xi_{p,J \setminus \{j\}} \right)\\
    &= \sum_{|J| = p+1} \int_{\Gamma_{J}} (\delta \xi_{p})_J \\
    &= \sum_{|J| = p+1} \int_{ \Gamma_{J}} \omega_{p,J}.
\end{align*}
Therefore the sum
$$
\sum_{|I| = p+1} \int_{ \Gamma_{I}} \omega_{p,I}
$$
is independent of $p$.  Evaluating the expression for $p=d-1$ and $p=0$ gives
\begin{align*}
    \left( \frac{1}{2 \pi i} \right)^d \int_{\Gamma} \omega &= \sum_{i = 1}^d \int_{\Gamma_i} \omega_{0,i} \\
    &= \sum_{i = 1}^d \int_{\Gamma_i} \eta_{\omega} \\
    &= \int_{\partial \Gamma_0} \eta_{\omega}, \quad \text{where } \Gamma_0 = \{\zz : |\textrm{det}(P_i(\zz))| \leq \epsilon \text{ for } 1 \leq i \leq d\} \\
    &= \int_{\partial U} \eta_{\omega}, 
\end{align*}
where in the last step we use Stokes' theorem and the fact that $d \eta_{\omega} = 0$ to justify $\int_{\partial \Gamma_0} \eta_{\omega} = \int_{\partial U} \eta_{\omega}$. This proves the identity in~\cref{eq:globalrescoho}.
\end{proof}

Consequently,  any representative of the Dolbeault class $\eta_\omega$ may be integrated over the boundary of a complex region to obtain the same residue as the local contour integral over $\Gamma$.  This observation is crucial for numerical computation, since it allows the residue to be evaluated using standard quadrature on $\partial U$. Moreover, residues associated with several eigenvalues may be obtained by integrating over regions containing multiple eigenvalues.

In the next result we derive an explicit representative of
$\eta_\omega$ that can be evaluated pointwise in $\C^d$. The construction generalizes~\cite[p.~653--654]{griffiths1978ag}; we adopt a slightly modified definition of the auxiliary functions $\rho_i$ to simplify the resulting sign conventions.

\begin{theorem} \label{thm:omegarep}
Let $\omega$ and $\eta_{\omega}$ be as in \cref{thm:globalrescoho}, and define
\begin{equation} \label{eq:rho}
\rho_i = (-1)^{i-1} (Q_i Q_i^H) \left(\sum_{j=1}^d Q_j Q_j^H \right)^{-1}, \quad 1 \leq i \leq d.\footnote{The definition of $\rho_i$ in~\cite{griffiths1978ag} does not include the alternating sign. We believe that there is a small sign error in~\cite[p.~654]{griffiths1978ag} when defining $\xi_{\{i,j\}^0},\omega_{\{i,j\}^0}$, which likely does not affect the final result, but makes the implied induction in their proof unclear. In particular, it is impossible with their definition to have $\delta \xi_{d-2} = \omega_{d-2}$. We find it simplest to reconcile the sign issue with this modified definition.}
\end{equation}
Then a representative of the Dolbeault class $\eta_\omega$ is given by 
\begin{equation} \label{eq:omegarep}
\eta_{\omega} \! = \! \left\{\omega_{\{k\}} \! = \! \left( \frac{1}{2 \pi i} \right)^{d}  \sum_{\sigma \in S_{d-1}} (-1)^{\text{sgn}(\sigma)}  \overline{\partial} \rho_{\sigma(1)} \wedge \cdots \wedge \widehat{\overline{\partial} \rho_{\sigma(k)}} \wedge \cdots \wedge \overline{\partial} \rho_{\sigma(d)} \wedge \omega \right\}_{k=1}^d \!\!,
\end{equation}
where $S_{d-1}$ denotes the symmetric group on $d-1$ elements and the hat indicates omitting the term $\overline{\partial} \rho_{\sigma(k)}$.\footnote{Here, and throughout the proof, we let $S_{d-1}$ act on $\{1,\ldots,d\} \setminus \{k\}$ in the natural way by shifting $q \mapsto q-1$ for $q > d$ and then using the usual action on $\{1,\ldots,d-1\}$. Equivalently, this can be thought of as letting $S_{d-1}$ act on the position in the set $\{1,\ldots,d\} \setminus \{k\}$ rather than the number.}
\end{theorem}
\begin{proof}
The matrices $\{\rho_i\}$ play the role of a partition of unity.  Indeed,
$$
\sum_{i=1}^d (-1)^{i-1} \rho_i = I_N.
$$
We use this identity to track explicit representatives of the image of $\omega$ through the chain of isomorphisms appearing in the proof of \cref{thm:dolbeault}. For the proof, we can ignore the constant $\left(\frac{1}{2 \pi i} \right)^{d}$, and just track $\omega$.  As in the proof of~\cref{thm:globalrescoho}, representatives of the intermediate cohomology groups are collections of differential forms defined on intersections of the sets of the open cover $\underline U$.

At each stage of the construction, we choose $\xi_{p+1} = \{\xi_{I^0} \in A^{d,d-p-2}(U_{I^0})\}_{|I| = d-p-1} \in C^{p}(\underline{U},A^{d,d-p-2})$ and $\omega_p = \{\omega_{I^0} \in Z^{d,d-p-1}(U_{I^0})\}_{|I| = d-p-1} \in H^{p}(\underline{U},Z^{d,d-p-1})$, where $U_{I^0} = \cap_{j \notin I} U_j$,\footnote{Notice that we switch the convention from the proof of \cref{thm:globalrescoho} to using $U_{I^0} = \cap_{j \notin I} U_j$ instead of $U_{I} = \cap_{i \in I} U_i$, which aligns with~\cite{griffiths1978ag}, and is cleaner for this proof.} satisfying $\delta \xi_p = \omega_p$ and $\overline{\partial} \xi_{p+1} = \omega_{p}$, where $\overline{\partial}$ is the Dolbeault operator and  $\delta$ denotes the sheaf cohomology coboundary operator as defined in \cref{eq:delta}. For $p=d-1,\ldots,0$ we define
\begin{align*}
\xi_{p+1} &= \left\{ \xi_{I^0}, \text{where }  \xi_{I^0} = \sum_{\sigma \in S_{d-p-1}} (-1)^{\text{sgn}(\sigma)}  \rho_{\sigma(i_1)} \overline{\partial} \rho_{\sigma(i_2)} \wedge \cdots \wedge \overline{\partial} \rho_{\sigma(i_{d-p-1})}  \wedge \omega \right\}, \\
\omega_p &= \left\{ \omega_{I^0}, \text{where } \omega_{I^0} = \sum_{\sigma \in S_{d-p-1}} (-1)^{\text{sgn}(\sigma)} \overline{\partial} \rho_{\sigma(i_1)} \wedge \cdots \wedge \overline{\partial} \rho_{\sigma(i_{d-p-1})} \wedge \omega \right\},
\end{align*}
where $I = \{i_1,\ldots,i_{d-p-1}\}$ and $i_1 < \ldots < i_{d-p-1}$, i.e.,  $\xi_{I^0}$ and $\omega_{I^0}$ depend only on $I$,  not on the order of the indices within $I$.  

The identity $\overline{\partial}\xi_{p+1}=\omega_p$ follows immediately from the fact that $\overline{\partial}$ distributes over the sum, the product rule, and the fact that $\overline{\partial} \circ \overline{\partial} = 0$.  It remains to prove $\delta\xi_{p+1}=\omega_{p+1}$.   We proceed by reverse induction on $p$, beginning with $p=d-1$.
With the above definitions,
\begin{align*}
\omega_{d-1} &= \omega \\
\xi_{d-1} &= \{ \xi_{\{i\}^0} = \rho_i \omega \}_{i =1}^d,\\
\omega_{d-2} &= \overline{\partial} \xi_{d-1} = \{ \omega_{\{i\}^0} =  \overline{\partial} \rho_i \wedge \omega  \}_{i =1}^d, \\
\xi_{d-2} &= \{ \xi_{\{i,j\}^0} = \rho_i \overline{\partial} \rho_j \wedge \omega - \rho_j \overline{\partial} \rho_i \wedge \omega \}_{i < j}, \\
 \omega_{d-3} &= \overline{\partial} \xi_{d-2} = \{ \omega_{\{i,j\}^0} = \overline{\partial} \rho_i  \wedge \overline \partial \rho_j \wedge \omega - \overline{\partial} \rho_j  \wedge \overline \partial \rho_i \wedge \omega\}_{i < j}.
\end{align*}
Hence, $\delta \xi_{d-1} = \omega_{d-1}$ by the alternating-sign convention in the definition of $\rho_i$.  A direct calculation using the partition identity for $\rho_i$ and the definition of $\delta$ shows that
\begin{align*}
(\delta \xi_{d-2})_{\{i\}^0} &= \sum_{j < i} (-1)^{j-1} \xi_{\{i,j\}^0} +  \sum_{j > i} (-1)^{j} \xi_{\{i,j\}^0} \\
&= \sum_{j < i} (-1)^{j-1} (-\rho_i \omega_{\{j\}^0} + \rho_j \omega_{\{i\}^0}) + \sum_{j > i} (-1)^{j-1}(-\rho_i \omega_{\{j\}^0} + \rho_j \omega_{\{i\}^0}) \\
&= -\rho_i \left( \sum_{j \neq i} (-1)^{j-1} \omega_{\{j\}^0} \right) + (I_N-(-1)^{i-1} \rho_i) \omega_{\{i\}^0} \\
&= -\rho_i \overline{\partial} \left( \sum_{j \neq i} (-1)^{j-1} \rho_j \omega \right) + (I_N-(-1)^{i-1} \rho_i) \omega_{\{i\}^0} \\
&= \rho_i \overline{\partial} \left( (-1)^{i-1} \rho_i \omega \right) + (I_N-(-1)^{i-1} \rho_i) \omega_{\{i\}^0} \\
&= (-1)^{i-1}\rho_i \omega_{\{i\}^0} + (I_N-(-1)^{i-1}\rho_i) \omega_{\{i\}^0} = \omega_{\{i\}^0},
\end{align*}
so $\delta \xi_{d-2} = \omega_{d-2}$, as desired.

For the general inductive step $\delta \xi_{p+1} = \omega_{p+1}$, we assume that the statement holds for $p' > p+1$. The steps are analogous to the above concrete proof that $\delta \xi_{d-2} = \omega_{d-2}$. Note that we can write
$$
\xi_{p+1} = \left\{ \xi_{I^0} =  \sum_{j=1}^{d-p-1} (-1)^{(j-1)}  \rho_{i_j}  \sum_{\sigma \in S_{d-p-2}} (-1)^{\text{sgn}(\sigma)}  \overline{\partial} \rho_{\sigma(i_2)} \wedge \cdots \wedge \overline{\partial} \rho_{\sigma(i_{d-p-1})} \wedge \omega  \right\}, \\
$$
with $I = \{i_1,\ldots,i_{d-p-1}\}$,
by separating the original sum according to which term is undifferentiated, and then projecting down to $S_{d-p-2}$. Precisely, define $a:\{1,\ldots,d-p-1\} \backslash \{1\} \to \{1,\ldots,d-p-2\}$ and $b:\{1,\ldots,d-p-1\} \backslash \{j\} \to \{1,\ldots,d-p-2\}$ by $a(n) = n-1$ and $b(n) = n$ if $n < j$ and $b(n) = n-1$ if $n > j$. Then any permutation $\sigma \in S_{d-p-1}$ with $\sigma(1) = j$ gives a permutation $b \circ \sigma \circ a^{-1} \in S_{d-p-2}$.
The change in sign follows from the definition of the sign of a permutation as the number of inversions in the set the permutation acts upon, ie. the number of pairs $x <y$ with $\sigma(x) > \sigma(y)$. It is clear that our map removes $j-1$ inversions associated to the elements of $\{1,\ldots,d-p-1\}$ that are less than $j$.

For $J = i_1,\ldots,i_{d-p-2}$, let $\ell_j, j =1,\ldots,p+2$ be the standard ordering of $\{1,\ldots,d\} \setminus J$. Then, using $\chi_{J^0}$ to denote the elements of $\delta \xi_{p+1}$, following the definition of the map $\delta$ and separating according to which term is undifferentiated gives
\begin{align*}
\delta \xi_{p+1} &  \!\!= \!\! \left\{ \chi_{J^0}   \!\!= \!\! \sum_{j = 1}^{p+2} (-1)^{j-1} \xi_{(J \cup \{\ell_j\})^0} \right\} \\
&  \!\!= \!\! \left\{ \chi_{J^0}   \!\!= \!\! \sum_{j = 1}^{p+2} (-1)^{j-1}  \!\!\!\!\!\! \sum_{\sigma \in S_{d-p-1}}  \!\!\!\!\! (-1)^{\text{sgn}(\sigma)} \rho_{\sigma(i_1)} \overline{\partial} \rho_{\sigma(i_2)} \! \wedge \! \cdots \! \wedge \! \overline{\partial} \rho_{\sigma(\ell_j)} \! \wedge \! \cdots \! \wedge \! \overline{\partial} \rho_{\sigma(i_{d-p-1})} \! \wedge \! \omega \right\} \\
&  \!\!= \!\! \left\{ \! \chi_{J^0}   \!\!= \!\! \sum_{j = 1}^{p+2} (-1)^{j-1} \! \sum_{k  = 1}^{d-p-1} \! (-1)^{(k-1)} \rho_{i_k}  \!\!\!\!\! \sum_{\sigma \in S_{d-p-2}}  \!\! \!\! (-1)^{\text{sgn}(\sigma)}  \overline{\partial} \rho_{\sigma(i_2)} \! \wedge \! \cdots \! \wedge \! \overline{\partial} \rho_{\sigma(i_{d-p-1})} \! \wedge \! \omega \! \right\} \\
&  \!\!= \!\! \left\{ \! \chi_{J^0}   \!\!= \!\! \sum_{j = 1}^{p+2} (-1)^{\ell_j-1} \rho_{\ell_j}  \! \!\! \left( \! \sum_{\sigma \in S_{d-p-2}}  \!\! \!\!\!\! (-1)^{\text{sgn}(\sigma)} \overline{\partial} \rho_{\sigma(i_2)} \! \wedge \! \cdots \! \wedge \! \overline{\partial} \rho_{\sigma(\ell_j)} \! \wedge \! \cdots \! \wedge \! \overline{\partial} \rho_{\sigma(i_{d-p-1})} \! \wedge \! \omega \right. \right. \\
&+ \left. \left. \sum_{j = 1}^{p+2} \sum_{\substack{k  = 1 \\ i_k \neq \ell_j}}^{d-p-1} (-1)^{j+k} \rho_{i_k}  \!\!\!\!\!\! \sum_{\sigma \in S_{d-p-2}}  \!\!\!\!\!(-1)^{\text{sgn}(\sigma)} \overline{\partial} \rho_{\sigma(i_2)} \! \wedge \! \cdots \! \wedge \! \overline{\partial} \rho_{\sigma(\ell_j)} \! \wedge \! \cdots \! \wedge \!\overline{\partial} \rho_{\sigma(i_{d-p-1})} \! \wedge \! \omega \!\! \right) \!\!\! \right\},
\end{align*}
where in the last step we have separated the terms with $i_k = \ell_j$ from the rest. The sign on $\rho_{\ell_j}$ comes from the fact that $j$ is the position of $\ell_j$ in $\{1,\ldots,d\} \setminus J$ and $k$ is the position of $\ell_j$ in $J \cup \{\ell_j\}$, whereas $\ell_j$ is the position in the full set $\{1,\ldots,d\}$, so $j+k=\ell_j+1$.
Focusing only on the second sum, and simplifying further, we have
\begin{align*}
& \sum_{j=1}^{p+2} \sum_{\substack{k = 1 \\ i_k \neq \ell_j}}^{d-p-1} (-1)^{j+k} \rho_{i_k} \sum_{\sigma \in S_{d-p-2} } (-1)^{\text{sgn}(\sigma)} \overline{\partial} \rho_{\sigma(i_2)} \wedge \cdots \wedge \overline{\partial} \rho_{\sigma(\ell_j)} \wedge \cdots \wedge \overline{\partial} \rho_{\sigma(i_{d-p-1})} \wedge \omega \\
&=  \sum_{k=1}^{d-p-1}  \rho_{i_k} \sum_{\substack{j=1 \\ \ell_j \neq i_k}}^{p+2} (-1)^{j+k}  \omega_{((J \setminus \{i_k\}) \cup \{\ell_j\} )^0}
\end{align*}
By the definition of $\delta$,
\begin{align*}
(\delta \omega_{p+1})_{(J \setminus \{i_k\})^0} \!\! = \!\! (-1)^{(i_k,(J \setminus \{i_k\})^0)}\omega_{((J \setminus \{i_k\}) \cup \{i_k\})^0} \!\! + \!\! \sum_{\substack{j=1 \\ \ell_j \neq i_k}}^{p+2} (-1)^{(\ell_j,J \setminus \{i_k\})^0 }  \omega_{((J \setminus \{i_k\}) \cup \{\ell_j\} )^0},
\end{align*}
where $(i_k,(J \setminus \{i_k\})^0)$ is the $0$-indexed position of $i_k$ in $\{1,\ldots d\} \setminus (J \setminus \{i_k\})$ and similarly for $(\ell_j,(J \setminus \{i_k\})^0)$. But, by the inductive hypothesis and the fact that $\delta$ and $\overline{\partial}$ commute, $\delta \omega_{p+1} = \delta \overline{\partial} \xi_{p+2} = \overline{\partial} \delta \xi_{p+2} =  \overline{\partial} \omega_{p+2} = 0$, so
\begin{align*}
\omega_{((J \setminus \{i_k\}) \cup \{i_k\})^0} &= \sum_{\substack{j=1 \\ \ell_j \neq i_k}}^{p+2} (-1)^{(\ell_j,(J \setminus \{i_k\})^0 + (i_k,(J \setminus \{i_k\})^0)-1}  \omega_{((J \setminus \{i_k\}) \cup \{\ell_j\} )^0} \\
&= \sum_{\substack{j=1 \\ \ell_j \neq i_k}}^{p+2} (-1)^{j+k+i_k+1}  \omega_{((J \setminus \{i_k\}) \cup \{\ell_j\} )^0}.
\end{align*}
Similarly to the analysis for the sign on $\rho_{\ell_j}$ above, we find the sign on $\omega_{((J \setminus \{i_k\}) \cup \{\ell_j\} )^0}$ by first noting that $j$ is the position of $\ell_j$ in $\{1,\ldots,d \} \setminus J$ and $k$ is the position of $i_k$ in $J \cup \{\ell_j\}$, with the indices starting at $1$. Suppose that $\ell_j < i_k$. Then $(\ell_j,(J \setminus\{i_k\})^0) = j+1$ and $(i_k,(J \setminus \{i_k\})^0) = i_k - k + 1$. Now suppose that $\ell_j > i_k$. Then $(\ell_j,(J \setminus\{i_k\})^0) = j+2$ and $(i_k,(J \setminus \{i_k\})^0) = i_k - k$. In either case, the sign becomes $(-1)^{j+k+i_k+1}$ with simple equivalences mod $2$.
Thus
\begin{align*}
&  \sum_{k=1}^{d-p-1}   \rho_{i_k} \sum_{\substack{j=1 \\ \ell_j \neq i_k}}^{p+2} (-1)^{j+k}  \omega_{((J \setminus \{i_k\}) \cup \{\ell_j\} )^0} \\
&=  \sum_{k=1}^{d-p-1}  (-1)^{i_k-1} \rho_{i_k} \omega_{J^0}
\end{align*}
Therefore
\begin{align*} 
\delta \xi_{p+1} &=
\left\{ \chi_{J^0} = \sum_{\ell_j \notin J} (-1)^{\ell_j-1} \rho_{\ell_j} \omega_{J^0} + \sum_{i_j \in J} (-1)^{i_j-1} \rho_{i_j} \omega_{J^0} \right\} \\
&= \left\{ \chi_{J^0}, \text{with } \chi_{J^0} =  \omega_{J^0} \right\} = \omega_{p+1}.
\end{align*}
This completes the inductive step. Therefore, 
$$
\eta_{\omega} \! = \! \left( \frac{1}{2 \pi i} \right)^d \!\! \omega_0 \! = \! \left\{\!\omega_{\{k\}} \!\! = \!\! \left( \frac{1}{2 \pi i} \right)^d \!\!\!\!\! \sum_{\sigma \in S_{d-1}} \!\!\!\! (-1)^{\text{sgn}(\sigma)} \overline{\partial} \rho_{\sigma(1)} \wedge \! \cdots \! \wedge \widehat{\overline{\partial} \rho_{\sigma(k)}} \wedge \! \cdots \! \wedge \overline{\partial} \rho_{\sigma(d)} \wedge \omega \!\right\}_{k=1}^d. 
$$
\end{proof}
The scalar version of this result in~\cite{griffiths1978ag} is considerably simpler.  In the scalar setting the coefficients of the differential forms commute,  so an expansion over the symmetric group is unnecessary.  Moreover, additional simplifications follow from the alternating property of scalar differential forms~\cite[p. ~654]{griffiths1978ag}, which are no longer available in the matrix-valued setting, since the wedge product of matrix-valued differential forms is not alternating. Nevertheless, although the expression for $\eta_{\omega}$ appears complicated, it already provides a concrete object that can be evaluated in practice. 

We have written $\eta_{\omega}$ as an element of $H^0(\underline{U}, Z^{d,d-1})$,  so it appears at first as a collection of $d$ differential forms, one on each element of the open cover $\underline{U}$. However, since $\delta \eta_{\omega} = 0$, the forms agree on overlaps: $\omega_{\{k\}} = \omega_{\{j\}}$ on $U_j \cap U_k$ for all $j,k$.  By the gluing property of sheaves, these local forms therefore define a single global differential form on $U^*$,  agreeing with $\omega_{\{k\}}$ on $U_k$, which we also denote $\eta_{\omega}$.  In practice, it is sufficient to evaluate any one of these expressions $\omega_{\{k\}}$, as all choices give the same formula. In \cref{sec:details} we further
simplify this expression and derive an explicit two-dimensional formula for $\eta_{\omega}$ that is symmetric with respect to the variables and the functions $P_i$.

The existence of a form that integrates over the boundary of a
$d$-dimensional complex region allows the residue construction to be globalized. In particular, we obtain an integral formula that can
capture several eigenvalues simultaneously. We now state the global residue theorem for matrix functions, generalizing
~\cite[p.~656]{griffiths1978ag}.

\begin{theorem} \label{thm:globalres}
Suppose that all eigenvalues $\zz^{(1)},\ldots,\zz^{(m)}$ of the MEP $P = \{P_i,1 \leq i \leq d\}$ in \cref{eq:mepform} lying in some compact region $C\subset\C^d$ are simple, with $C$ having smooth boundary. Let $g$ and $\omega$ be as in \cref{thm:globalrescoho}, and let
$
\eta_{\omega}
$
be the representative in \cref{thm:omegarep}. Then
\begin{equation} \label{eq:globalres}
	\int_{\partial C} \eta_{\omega} = \sum_{i=1}^{m} \operatorname{Res}_P(g,\zz^{(i)})
\end{equation}
\end{theorem}
\begin{proof}
The argument follows the same reasoning as in~\cite[p. ~656]{griffiths1978ag}.  Let $\epsilon>0$ be sufficiently small so that the balls $B_{\epsilon}(\zz^{(i)})$ are disjoint and contained in $C$.  Applying Stokes' theorem to the region 
$
C \setminus \bigcup_{i=1}^{m} B_{\epsilon}(\zz^{(i)}),
$
and using that $d\eta_{\omega}=0$, we obtain
$$
\int_{\partial C} \eta_{\omega} - \sum_{i=1}^{m} \int_{\partial B_{\epsilon}(\zz^{(i)})} \eta_{\omega} = \int_{C \backslash (\cup_{i=1}^{m} B_{\epsilon}(\zz^{(i)}))} d \eta_{\omega} = 0,
$$
Hence, 
\[
\int_{\partial C} \eta_{\omega}
=
\sum_{i=1}^{m}
\int_{\partial B_{\epsilon}(\zz^{(i)})} \eta_{\omega}.
\]
Each boundary integral equals the local residue by 
\cref{thm:globalrescoho,thm:omegarep}, which yields
\begin{align*}
\int_{\partial C} \eta_{\omega} &= \sum_{i=1}^{m} \int_{\partial B_{\epsilon}(\zz^{(i)})} \eta_{\omega} \\
&= \sum_{i=1}^{m} \operatorname{Res}_P(g,\zz^{(i)}).
\end{align*}
\end{proof}

\Cref{thm:globalres}  completes the theoretical foundation for our contour integral formulation. We can now calculate the sum of residues over an arbitrary number of eigenvalues inside any region $C$. In particular, we no longer require any knowledge of the eigenvalues in order to calculate the residue. Combined with \cref{thm:simplematrixgrothres}, this allows us to convert a boundary integral over $\partial C$ into a finite sum of rank one contributions that are directly tied to the eigenvalues.  In \cref{sec:alg}, we prove that if we calculate the residue for sufficiently many distinct functions $g$, we can find the eigenvalues. This is analogous to both univariate contour eigensolvers such as Beyn's method and to scalar contour solvers.

This completes our extension of the basic elements of the scalar Grothendieck residue theory from~\cite{griffiths1978ag} to the matrix
multiparameter setting. In particular, we have developed both a local
and a global residue theory sufficient to support the contour
integral method introduced in this paper. Our construction also
generalizes the residue used in the scalar contour methods of
\cite{kravanja1998contour,kravanja2000contour}, although our approach
does not proceed through a Bochner--Martinelli type formula.
Consequently, our method may be viewed as a matrix extension
of the scalar contour techniques developed in
\cite{kravanja1998contour,kravanja2000contour}.

From the perspective of numerical linear algebra, our results also
extend the contour-integral identity underlying univariate contour eigensolvers. Consequently, the framework developed here can be viewed as a natural multiparameter extension of Beyn's method~\cite{asakura2009beyn,beyn2012beyn,yokota2013beyn}.
However, we do not attempt to generalize Keldysh's theorem
(\cref{thm:Keldysh}; see also~\cite[Thm.~1.6.5]{mennicken2003keldysh}).
Such a generalization is likely unrelated to the residue formula
required for a contour solver.
Keldysh's theorem describes expansions in a neighborhood of a pole, which in one variable is directly related to a residue at that pole. In several complex variables the singular sets of analytic functions are typically higher-dimensional hypersurfaces rather than isolated poles, while residues remain associated with isolated points. Consequently, the natural multivariate analogue of a pole expansion is more closely related to the Weierstrass preparation theorem, even in the scalar case. Some matrix-valued extensions of Weierstrass
preparation have been studied~\cite{dencker1993preparation,balicki2026parametric}, but these results do not directly connect with the residue framework developed here. Accordingly, there is little reason to expect that a multidimensional
contour solver should arise from a generalization of Keldysh's theorem.

\section{MultiBeyn: A Multiparameter Contour Method}
\label{sec:alg}

We now present a multidimensional contour algorithm for multiparameter eigenvalue problems (MEPs), which we call multiBeyn as it generalizes Beyn's method~\cite{asakura2009beyn,beyn2012beyn,yokota2013beyn} to several complex variables. Throughout this section we consider an MEP with $d$ analytic matrix functions in $d$ variables. Implementation details for the practically important two-parameter case are given in \cref{sec:details}.  We assume throughout that we have some reasonable quadrature rule for calculating the contour integrals.  In \cref{subsec:quad}, we describe an efficient quadrature rule for the two-parameter case by simplifying~\cref{eq:omegarep}. 

Our core idea is to combine the global residue formula from \cref{sec:theory} with moment-based contour techniques, yielding a reduced eigenvalue problem whose size is proportional to the number of eigenvalues inside the region of interest (see~\cref{subsec:mom}).  At a high level,  the method proceeds by integrating the matrix-valued differential form $\eta_{\omega}$ over the boundary $\partial C$ to extract spectral information about the MEP inside $C$.  By weighting this integral with powers of a selected coordinate (here $z_d$), we obtain a sequence of moments that encode the eigenvalues in that coordinate.  These moments are assembled into a Hankel matrix pencil whose eigenvalues coincide with the desired $z_d$-coordinates. The remaining coordinates are then recovered by solving lower-dimensional MEPs obtained by fixing previously computed components. 

\begin{algorithm}
\textbf{multiBeyn: Contour method for analytic MEPs.}

\hrule

\vspace{.1cm}

\textbf{Input:} MEP in \cref{eq:mepform}, compact region $C \subset \C^d$ with smooth boundary, sketching matrices $\hat{V},\hat{W} \in \C^{N \times k}$ and moment matrix size $M$ with $kM \ge m$, where $m$ is the number of eigenvalues in $C$.
\begin{algorithmic}[1]
\State Calculate $\hat{V}^H \left( \int_{\partial C} \eta_{\omega(z_d^i)}(\zz) \right) \hat{W}$ for $i=0,\ldots,2M-1$.\footnotemark
\State Form the Hankel eigenvalue problem in \cref{eq:Hankel}.
\State Solve the Hankel eigenvalue problem using QZ for $z_d^*$.
\State Recover the remaining coordinates by fixing $z_d$ at computed values,  and repeating steps $1$-$3$ on MEPs with dimension $d-1$.
\end{algorithmic}
\textbf{Output:} All eigenvalues $(z_1^*,\ldots,z_d^*)$ of \cref{eq:mepform} contained in $C$.
\label{alg:mpe}
\end{algorithm}

We first describe how to compute the $d$-th coordinate $z_d$.  Given the residue formula in \cref{sec:theory}, this reduces to a univariate contour problem. Consequently, much of the algorithm parallels standard contour eigensolvers, and inherits the same design tradeoffs.  A direct analogue of Beyn's method~\cite{beyn2012beyn} computes 
\begin{equation} \label{eq:naivemethod}
A_i = \int_{\partial C} \eta_{\omega(z_d^i)}(\zz), \quad i = 0,1,
\end{equation}
and then solves the linear eigenvalue problem $(A_1 - z_d A_0) \vv = 0$.   The differential form $\eta_{\omega}$ in \cref{eq:omegarep} consists of sums of products of noncommuting matrix factors and commuting differential forms $dz_j$ and $d\overline{z}_j$. This structure allows the boundary integral over $\partial C$ to be evaluated by standard multidimensional quadrature. In \cref{subsec:quad} we derive explicit quadrature formulas for the two-parameter case.  As in the univariate setting~\cite{guttel2017nlevp,imakura2016survey}, two key design choices determine the efficiency of the method: (1) the use of sketching to reduce the problem dimension and (2) the number of moments used in the Hankel construction.

\footnotetext{Recall that $\omega$ depends on a scalar function $g$; $\eta_{\omega(z_d^i)}(\zz)$ denotes $\eta_{\omega}$ with $g = z_d^i$.}

\subsection{Sketching}

Sketching is essential for making the contour integral computationally tractable. 
A direct evaluation of the moments
$
A_i = \int_{\partial C} \eta_{\omega(z_d^i)}(\zz)
$
requires manipulating matrices of size 
$
N = \prod_{i=1}^d n_i,
$
which is prohibitively expensive even for moderate $d$.  To avoid this, we work with a projected version of the moment matrices. 
Let $\hat V,\hat W \in \C^{N \times k}$ with $k \ll N$. We define the sketched residue
\begin{equation} \label{eq:lessnaivemethod}
\mu_i = \hat V^H \left( \int_{\partial C} \eta_{\omega(z_d^i)}(\zz) \right) \hat W,
\quad i \ge 0.
\end{equation}
These are $k \times k$ matrices that encode the same spectral information, provided the subspaces spanned by $\hat V$ and $\hat W$ capture the eigenstructure inside $C$.  More precisely, let $m$ be the number of eigenvalues in $C$, and let $V,W \in \C^{N \times m}$ denote the matrices of right and left eigenvectors (in Kronecker form). If
$
\operatorname{rank}(\hat V^H V) = m
$
and 
$
\operatorname{rank}(W^H \hat W) = m,
$
then the projected pencil constructed from $\{\mu_i\}$ yields exactly the eigenvalues inside $C$. In practice, these conditions are satisfied with high probability by choosing $\hat V$ and $\hat W$ as random matrices with $k \gtrsim m$.

This projection yields two key advantages: (1) The resulting eigenvalue problem has dimension $\mathcal{O}(m)$ rather than $N$ and (2) The contour integral can be evaluated without forming $\eta_{\omega}$ explicitly.  Indeed, at each quadrature point the integrand reduces to quantities of the form
$
\hat V^H \eta_{\omega}(\zz) \hat W,
$
which can be evaluated by applying structured operators to a small number of vectors. In particular, one only needs to compute expressions of the form
$
\mathbf{v}^H \eta_{\omega}(\zz) \mathbf{w},
$
which decompose into structured matrix products and linear solves involving the original matrices $P_i(\zz)$. This avoids the explicit construction of the Kronecker matrices $Q_i$ entirely.  The resulting structure is especially favorable in low dimensions. In \cref{sec:details} we show that, in the bivariate case, this leads to a substantial reduction in both computational cost and memory usage.

\subsection{Moment-Based Method for the Last Coordinate of the Eigenvalues}
\label{subsec:mom}

The projected residue \eqref{eq:lessnaivemethod} can be used more effectively by computing a sequence of moments (as in~\cite{sakurai2003moments}).  This leads to a structured eigenvalue problem whose size depends on the number of eigenvalues in $C$, while reducing the number of expensive contour evaluations.  Let $M \in \mathbb{N}$, which will be the number of $k \times k$ blocks in our moment matrices in \cref{eq:Hankel}. We define the sketched moments
\begin{equation}
\mu_i = \hat V^H \left( \int_{\partial C} \eta_{\omega(z_d^i)}(\zz) \right) \hat W,
\quad i = 0,1,\ldots,2M-1.
\end{equation}
As $M$ increases,  one can select a smaller sketching dimension $k$, reducing the number of linear solves per quadrature point as well as improving the robustness of the eigensolver when the number of eigenvalues $m$ is large. We find that the additional cost of increasing $M$ is negligible,  since higher moments are obtained by multiplying previously computed quantities by $z_d$ (a similar phenomenon happens for DAEs~\cite{bai2002dae,freund2003dae,grimme1997dae}).  From the moments $\{\mu_i\}$ we form $(Mk)\times(Mk)$ block Hankel matrices given by 
\begin{equation}  \label{eq:Hankel}
H_{M} =
\begin{bmatrix}
\mu_0 & \mu_1 & \cdots & \mu_{M-1} \\
\mu_1 & \mu_2 & \cdots & \mu_{M} \\
\vdots & \vdots & \ddots & \vdots \\
\mu_{M-1} & \mu_{M} & \cdots & \mu_{2M-2}
\end{bmatrix}, \qquad
H_{M}^{<} =
\begin{bmatrix}
\mu_1 & \mu_2 & \cdots & \mu_{M} \\
\mu_2 & \mu_3 & \cdots & \mu_{M+1} \\
\vdots & \vdots & \ddots & \vdots \\
\mu_{M} & \mu_{M+1} & \cdots & \mu_{2M-1}
\end{bmatrix}.
\end{equation}
We then consider the generalized eigenvalue problem
$
(H_M^{<} - z_d H_M)\mathbf{y} = 0.
$

We now show that this problem recovers the desired eigenvalues. 
Let $m$ be the number of eigenvalues in $C$, and let
$
V = [\vv^{(1)} \cdots \vv^{(m)}]\in \C^{N \times m}
$
and 
$
W = [\ww^{(1)} \cdots \ww^{(m)}] \in \C^{N \times m}
$
be the matrices of right and left eigenvectors,  where each $\vv^{(i)}$ and $\ww^{(i)}$ is a Kronecker product as in \cref{sec:theory}.   Let
$
\Lambda = \mathrm{diag}(z_d^{(1)},\ldots,z_d^{(m)})
$
contain the $z_d$-coordinates of the eigenvalues in $C$, which may not all be distinct.  By \cref{sec:theory},  the moments admit the factorization
\[
\mu_i = \hat{V}^H V \Lambda^i W^H \hat{W}.
\]

Define the block Vandermonde matrices
\[
\mathcal{V} =
\begin{bmatrix}
\hat{V}^H V \\
\hat{V}^H V \Lambda \\
\vdots \\
\hat{V}^H V \Lambda^{M-1}
\end{bmatrix}
\in \C^{Mk \times m},
\quad
\mathcal{W} =
\begin{bmatrix}
W^H \hat{W} &
\Lambda W^H \hat{W} &
\cdots &
\Lambda^{M-1} W^H \hat{W}
\end{bmatrix}
\in \C^{m \times Mk}.
\]
Then, we have 
$
H_M = \mathcal{V}\mathcal{W},
$
and
$
H_M^{<} = \mathcal{V} \Lambda \mathcal{W}.
$
It follows that
$
H_M^{<} - z_d H_M = \mathcal{V}(\Lambda - z_d I)\mathcal{W},
$
and hence,  the generalized eigenvalues are precisely $z_d^{(1)},\ldots,z_d^{(m)}$ (up to possible singularities from overparameterization).
In practice, $H_M$ is rank-deficient with rank $m \ll Mk$. We therefore compute a truncated SVD, i.e., 
$
H_M = V_0 \Sigma_0 W_0^H,
$
where $\Sigma_0 \in \mathbb{R}^{m \times m}$ contains the nonzero singular values (above a chosen tolerance). Projecting yields the reduced eigenvalue problem
\begin{equation} \label{eq:finalevp}
V_0^H H_M^{<} W_0 \mathbf{y}
=
z_d \, V_0^H H_M W_0 \mathbf{y},
\end{equation}
which is of size $m \times m$.  Since
$
V_0^H H_M^{<} W_0 - z_d V_0^H H_M W_0
=
V_0^H \mathcal{V} (\Lambda - z_d I)\mathcal{W} W_0,
$
the eigenvalues of \eqref{eq:finalevp} are exactly the $d$-th coordinates of the eigenvalues inside $C$.

\subsection{Computing the Other Coordinates of the Eigenvalues}
Once the $z_d$-coordinates have been computed, the remaining coordinates $z_1,\ldots,z_{d-1}$ can be recovered in several ways.   One approach, analogous to methods in algebraic geometry and scalar contour techniques~\cite{graf2025pmep,kravanja1998contour,kravanja2000contour}, is to augment the moment construction.  For example, one may compute mixed moments of the form $z_j z_d^i$ for $j=1,\ldots,d-1$, and then solve an additional eigenvalue problem together with a matching step to pair the coordinates correctly.  While this approach is viable,  it introduces additional algebraic overhead.

A more direct alternative is to exploit the structure of the multiparameter problem.  Given a computed value $z_d^*$, we substitute $z_d = z_d^*$ into the functions $P_1,\ldots,P_d$ and obtain a reduced $(d-1)$-parameter eigenvalue problem. The multiBeyn procedure can then be applied recursively to this lower-dimensional problem.   In practice, it is not necessary to consider all $(d-1)$-tuples of the functions $P_i$.  Any eigenvalue that is shared across a collection of subsets whose union contains all indices $\{1,\ldots,d\}$ must be a solution of the full system. Consequently, only a small number of such subsets are required.  Empirically, using two randomly chosen subsets is sufficient to recover the remaining coordinates, and in general no more than $d$ subsets are needed per coordinate.

The computational cost of these subproblems is negligible. Each recursive step reduces both the dimension of the parameter space and the size of the matrices involved. For example, in the two-parameter case, the initial contour computation requires orders of magnitude more quadrature points than the subsequent one-dimensional problems, while also involving larger matrices. As a result, the overall cost is dominated by the first (highest-dimensional) solve.  In our experiments, the recursive subproblem approach is consistently the most robust, and it is therefore the strategy adopted in our implementation.

\section{Implementation Details of multiBeyn for $d =2$}
\label{sec:details}

To demonstrate the practical performance of multiBeyn, we implement a two-parameter version, which we call biBeyn~\cite{graf2026beyn}. In this section we describe the key ingredients required for an efficient implementation. Numerical results are presented in \cref{sec:num}.  Throughout this section we assume the MEP (see~\cref{eq:mepform}) is given by two analytic matrix functions $P_1 \in \Omega(\C^2,\C^{n_1 \times n_1})$ and $P_2 \in \Omega(\C^2,\C^{n_2 \times n_2})$.

\subsection{Simplifying $\eta_{\omega}$ in Two Variables}
\label{subsec:etaomega2d}

The main objective is to obtain a representation of $\eta_{\omega}$\footnote{We focus on simplifying $\eta_{\omega}$ for $g=1$. It is then trivial to multiply by whatever scalar $g$ we want to use.} that can be evaluated efficiently at quadrature points. As emphasized in \cref{sec:theory},  a practical implementation must avoid explicitly forming the Kronecker matrices $Q_i$, which are of size $N = n_1 n_2$.

Recall from \cref{thm:omegarep} that $\eta_{\omega}$ can be written locally as
$
\eta_{\{1\}} = \overline{\partial} \rho_{2}  \wedge \omega$ and  $ 
\eta_{\{2\}} = \overline{\partial} \rho_{1} \wedge \omega,
$
with
$
\omega(P_1,P_2) = (Q_1Q_2)^{-1} dz_1 dz_2.
$
Since these representations agree wherever both $P_1$ and $P_2$ are nonsingular, we are free to choose either one; we focus on simplifying $\eta_{\{2\}}$.
To expose the computational structure, it is convenient to work instead with the form $\eta_{\omega}'$, with representatives $\eta_{\{2\}}' = \omega \wedge \overline{\partial} \rho_1$ and $\eta_{\{1\}}' = \omega \wedge \overline{\partial} \rho_2$. This has the effect of distributing inverse factors more evenly across the expression, leading to a more numerically favorable representation. 

To confirm that this yields the same integral formula, we use the coboundary operator $\delta$ and the punctured domain $U^*$ from \cref{sec:theory}. Let $\xi_1 = \{ \xi_{\{i\}^0} = \rho_i \omega\}_{i=1}^2$ and $\xi_1' = \{ \xi_{\{i\}^0}' =  \omega \rho_i\}_{i=1}^2$.
Note that $\delta (\xi_1 - \xi_1') = \omega - \omega = 0$, so $\xi_1-\xi_1'$ has a globally defined representative in $H^0(A^{2,0}(U^*,\C^{N \times N})) = A^{2,0}(U^*,\C^{N \times N})$. And
$\overline{\partial} (\xi_1-\xi_1') = \eta_{\omega} - \eta_{\omega}'$, 
so 
$\eta_{\omega} = \eta_{\omega}'$ 
in the Dolbeault cohomology group 
$H^{d,d-1}(U^*,\C^{N \times N})= \frac{Z^{2,1}(U^*,\C^{N \times N})}{\overline{\partial} A^{2,0}(U^*,\C^{N \times N})}$. Therefore $\eta_{\{2\}} = \eta_{\{2\}}' = \omega \wedge \overline{\partial} \rho_1$ is also a reasonable representative to use in our algorithm.\footnote{This idea extends to any number of variables. In fact, the placement of $\omega$ in the sequence of wedge products in \cref{eq:omegarep} has no effect on the result other than possibly on the sign of the integral.}

With $\rho_i$ defined in \cref{eq:rho}, we obtain\footnote{While we focus on $d = 2$ in this section, one can use an analogous expansion to rewrite $\eta_{\omega}$ as a sum of concrete matrix-valued differential $(d,d-1)$ forms for any $d\geq 2$,  allowing it to be evaluated by numerical quadrature on the boundary of a $d$-dimensional complex region. In \cref{subsec:triBeyn} we sketch the formulas for a trivariate implementation.}
\begin{equation} \label{eq:dbrho}
\overline{\partial} \rho_i = (-1)^{i-1} \left(Q_i (d Q_i)^H  -  (Q_iQ_i^H)D^{-1}\left(Q_1 (d Q_1)^H  +  Q_2 (d Q_2)^H\right)\right) D^{-1},
\end{equation}
where $D = Q_1 Q_1^H + Q_2 Q_2^H$.  Transforming \cref{eq:dbrho} yields 
\begin{align*}
    Q_1Q_2 \eta_{\{2\}} &=  \left(Q_1 (d Q_1)^H - (Q_1Q_1^H)D^{-1}\left(Q_1 (d Q_1)^H + Q_2 (d Q_2)^H\right)\right) D^{-1}dz_1 dz_2 \\
&= \left( \left(I_N - Q_1Q_1^HD^{-1} \right)Q_1(d Q_1)^H   - Q_1Q_1^HD^{-1} Q_2 (d Q_2)^H\right) D^{-1}dz_1 dz_2 \\
&= \left( \left(Q_2Q_2^HD^{-1} \right)Q_1(d Q_1)^H  -  Q_1Q_1^HD^{-1} Q_2 (d Q_2)^H\right) D^{-1}dz_1 dz_2.
\end{align*}
Therefore, using $\eta_{\{2\}} = \omega \wedge \overline{\partial} \rho_1$ as our representative, we have
\begin{align*}
    \eta_{\omega} &=  \left( Q_1^{-1}Q_2^HD^{-1} Q_1 d Q_1^H - Q_1^HQ_2^{-1}D^{-1}Q_2 d Q_2^H \right) D^{-1} dz_1 dz_2 \\
    &=  \left( Q_2^H(Q_1^HQ_1 + Q_2Q_2^H)^{-1} d Q_1^H  \right.
    \left. - Q_1^H(Q_1Q_1^H + Q_2^HQ_2)^{-1} d Q_2^H \right) D^{-1} dz_1 dz_2
\end{align*}
We show in \cref{subsec:sylv} how we can use this representation to avoid forming Kronecker products explicitly.  For implementation it is convenient to decompose $\eta_{\omega}$ into two $(2,1)$-forms, given by $\eta_{\omega} = U_1 + U_2$, where
\begin{align*}
    U_2 &= \left( Q_2^H(Q_1^HQ_1 + Q_2Q_2^H)^{-1}  \frac{\partial Q_1^H}{\partial z_1} - Q_1^H(Q_1Q_1^H + Q_2^HQ_2)^{-1} \frac{\partial Q_2^H}{\partial z_1}\right) D^{-1} d\overline{z}_1 dz_1 dz_2, \\
    U_1 &= \left( Q_2^H(Q_1^HQ_1 + Q_2Q_2^H)^{-1}  \frac{\partial Q_1^H}{\partial z_2} - Q_1^H(Q_1Q_1^H + Q_2^HQ_2)^{-1} \frac{\partial Q_2^H}{\partial z_2}\right) D^{-1} d\overline{z}_2 dz_1 dz_2.
\end{align*}
Thus, the moment computation in biBeyn reduces to evaluating the moments given by 
    $$
    \frac{1}{(2 \pi i)^2} \int_{\partial C} \hat{V}^H z_d^j (U_1+U_2) \hat{W}, \quad j = 0,\ldots, 2M-1. 
    $$
    The advantage of this formulation is that each term can be evaluated using only matrix-vector products and linear solves involving $P_1$ and $P_2$, without ever forming the full Kronecker operators.   In the following sections we describe how to evaluate $U_1$ and $U_2$ efficiently at quadrature points and how to combine this with suitable quadrature rules.

\subsection{Evaluating the Residue Integrand in Two Variables}
\label{subsec:sylv}

We now describe how to evaluate the integrand efficiently. Without loss of generality, we focus on $U_1$. As emphasized earlier, the key requirement is to avoid forming the Kronecker matrices $Q_i$ explicitly.  Instead, all computations are performed by exploiting the implicit tensor-product structure.  The evaluation proceeds by applying $U_1$ to vectors.  Since $\hat V$ and $\hat W$ have only $k \ll N$ columns, it suffices to compute quantities of the form
$
\vv^H U_1 \ww,
$
for vectors $\vv \in \C^N$ and $\ww \in \C^N$. The full matrix $\hat V^H U_1 \hat W$ is then assembled from $k^2$ such evaluations.

Inspecting the expression for $U_1$, we identify three fundamental types of operations: (1) left multiplications, i.e., $\vv^H Q_i^H$, (2) applications of derivatives, i.e.,  $\frac{\partial Q_i^H}{\partial z_j} \mathbf{x}$, and (3) applications of inverse sums, i.e.,  $(Q_1 Q_1^H + Q_2 Q_2^H)^{-1} \ww$.   To exploit the structure,  we generate our sketching vectors to be separable. That is, 
\[
\vv = \vv_1 \otimes \vv_2, \qquad \ww = \ww_1 \otimes \ww_2,
\]
with $\vv_i,\ww_i \in \C^{n_i}$.  Here, $\otimes$ is the Kronecker product operation.  This restriction preserves the rank conditions required in \cref{subsec:mom} while enabling efficient computations.

\paragraph{(1) Structured multiplications.}
Products with $Q_i$ and $Q_i^H$ decouple across factors, i.e., we have
\[
\vv^H Q_2^H = (\vv_1^H \otimes \vv_2^H)(I_{n_1} \otimes P_2^H)
= \vv_1^H \otimes (\vv_2^H P_2^H).
\]
Thus, these operations reduce to matrix-vector products with $P_i$ rather than with $Q_i$.   Moreover, the output remains separable: applying $Q_i$ or $Q_i^H$ to a separable vector yields another separable vector and so the structure propagates through the computation to the next operation.

\paragraph{(2) Derivative terms.}
For derivative terms,  the intermediate vectors are no longer separable.  To avoid forming Kronecker products explicitly, we use the identity
$
(A \otimes B)\,\mathrm{vec}(X) = \mathrm{vec}(B X A^T).
$
For example,
\[
\frac{\partial Q_1^H}{\partial z_1} \mathbf{x}
=
\left(\frac{\partial P_1^H}{\partial z_1} \otimes I_{n_2}\right)\mathrm{vec}(X)
=
\mathrm{vec}\!\left(X \left(\frac{\partial P_1^H}{\partial z_1}\right)^T\right),
\]
where $\mathrm{vec}$ takes a matrix and stacks its columns one-by-one to make a vector.  Thus,  derivative terms can also be computed via a matrix multiplication of size $n_i \times n_i$.

\paragraph{(3) Inversion of Kronecker sums.}
The term $(Q_1 Q_1^H + Q_2 Q_2^H)^{-1}$ is the inverse of a Kronecker sum.  In general, this would require solving a Sylvester equation, for instance via the Bartels--Stewart algorithm.  However, in our setting, the matrix equation simplifies significantly.   Since $P_i P_i^H$ is Hermitian positive semi-definite,  its Schur decomposition diagonalizes it, i.e., 
$
P_i P_i^H = \mathcal Q_i U_i \mathcal Q_i^H.
$
It follows that $\mathcal Q_1 \otimes \mathcal Q_2$ simultaneously diagonalizes both $Q_1 Q_1^H$ and $Q_2 Q_2^H$. Transforming into this basis reduces the action of $(Q_1 Q_1^H + Q_2 Q_2^H)^{-1}$ to an elementwise division, after which we transform back (this is classical~\cite[Ch.~12]{golub2013matrix}, and also similar to~\cite{benzi2017kron}).   Thus, the cost of this step is dominated by basis transformations and avoids solving a full Sylvester system.  An analogous procedure applies when we have the term $(Q_1^H Q_1 + Q_2^H Q_2)^{-1}$.

\paragraph{Assembly.}
Our evaluation proceeds by propagating the right vector $\ww$ and the left vector $\vv$ inward through the expression for $U_1$,  exploiting separability wherever possible. The final result is obtained by taking the inner product of the resulting vectors.  For $k$ columns of $\hat V$ and $\hat W$, this requires $k^2$ such evaluations, each involving only structured matrix-vector products and small linear solves.  This approach avoids forming the Kronecker matrices entirely and reduces all computations to operations on the original matrices $P_1$ and $P_2$, which is essential for scalability.

\subsection{Quadrature Rules}
\label{subsec:quad}

The procedure described so far can be applied to any region $C \subset \C^2$,  and in principle any quadrature rule on $\partial C$ can be used.  For our applications, we choose $C$ to be a product of ellipses (a polyellipse), which provides greater flexibility than a polydisc (as in~\cite{kravanja1998contour,kravanja2000contour}),  and, importantly, allows us to localize the search near the real plane (see~\cref{sec:num}). 

Let $z_i = x_i + i y_i$, and define
\[
D = D_1 \times D_2,
\qquad
D_i = \{ x_i^2 + y_i^2/b^2 < 1 \}.
\]
Each $D_i$ is an ellipse with major axis $1$ along the real axis and minor axis $b$ along the imaginary axis. By scaling and translation, this construction can be adapted to arbitrary rectangular regions in $\mathbb{R}^2$.  It is tricky to visualize this region because it is in four dimensions,  but it is the standard Cartesian product of two ellipses. We give a projection in \cref{fig:polyellipse} that demonstrates how the polyellipse encloses a real square.

\begin{figure}
    \centering
    \begin{minipage}{0.55\textwidth}
        \begin{overpic}[width=\textwidth, tics=10]{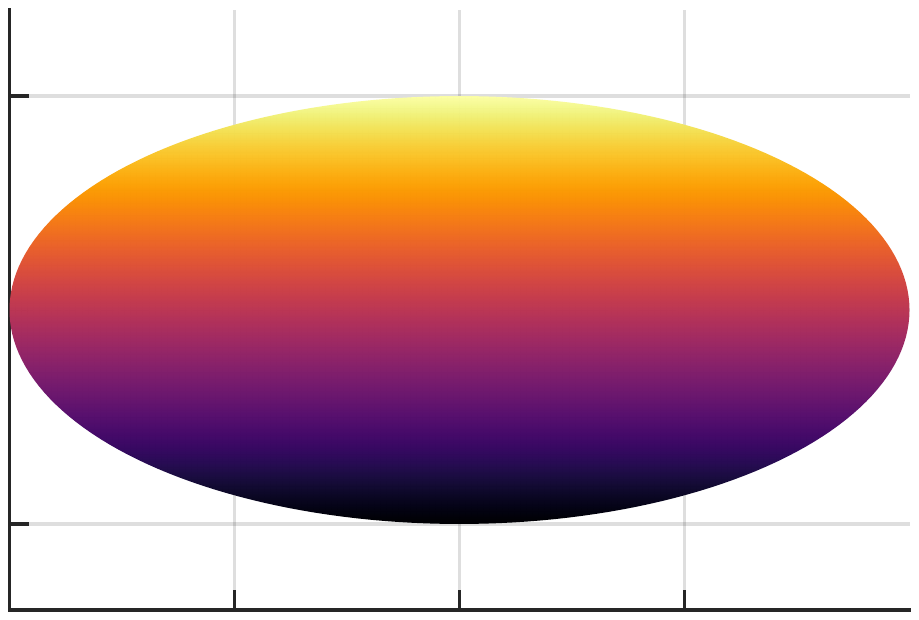}
            \put(35,25){$\text{Re}(z_1)$}
            \put(3,45){\rotatebox{90}{$\text{Im}(z_1)$}}
             {\footnotesize
            \put(6,36){$-.5$}
            \put(9.5,50){$0$}
            \put(8.5,62){$.5$}
            \put(9,29){$-1$}
            \put(22,29){$-.5$}
            \put(38,29){$0$}
            \put(51,29){$.5$}
            \put(65,29){$1$}
            }
        \end{overpic}
    \end{minipage}
    \hspace{-1.5cm}
    \begin{minipage}{0.55\textwidth}
        \begin{overpic}[width=\textwidth]{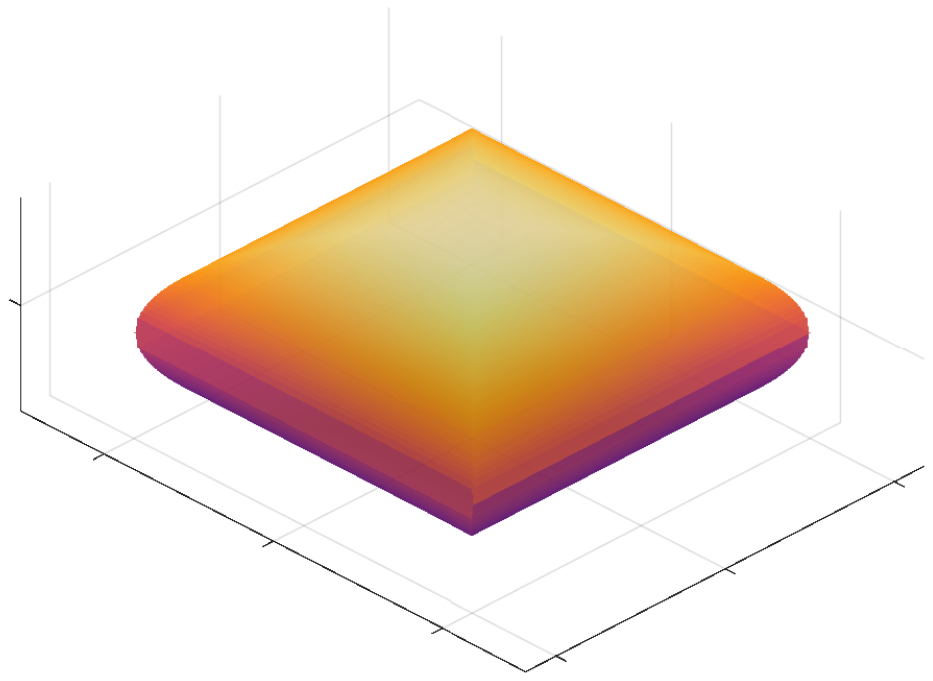}
            \put(15,30){$\text{Re}(z_1)$}
            \put(60,30){$\text{Re}(z_2)$}
            \put(5,35){\rotatebox{90}{$\min(\text{Im}(z_1), \text{Im}(z_2))$}}
            {\footnotesize
            \put(10,52){$0$}
            \put(13,41){$-1$}
             \put(25.5,36){$0$}
             \put(35,31){$1$}
            \put(45,29){$-1$}
            \put(56,34){$0$}
            \put(66,39){$1$}
            }
        \end{overpic}
    \end{minipage}
    
    \vspace{-2cm}
    \caption{Left: a one complex-dimensional slice of our polyellipse $D$ with $b=0.5$. Right: A projection of $D$ into $3$ dimensions.}
    \label{fig:polyellipse}
\end{figure}

The moment computation requires evaluating
\[
\frac{1}{(2\pi i)^2} \int_{\partial D} \hat{V}^H z_2^j (U_1 + U_2)\hat{W}, \qquad j=0,\ldots,2M-1.
\]
The boundary of a polyellipse decomposes as $\partial D = (D_1 \times \partial D_2) \cup (\partial D_1 \times D_2).$ Since $U_1$ and $U_2$ have complementary support, we treat them separately. We focus on $U_1$, for which $ \int_{D_1 \times \partial D_2} \hat{V}^H z_2^j U_1 \hat{W} = 0, $ for all $j$,  so only the contribution from $\partial D_1 \times D_2$ remains, i.e., 
\[
\int_{\partial D_1 \times D_2} \hat{V}^H z_2^j U_1 \hat{W}, \qquad j=0,\ldots,2M-1.
\]
The situation for $U_2$ is analogous, with the roles of the variables reversed.

We parameterize $\partial D_1 \times D_2$ using polar coordinates, i.e., 
$
z_1 = \cos(2\pi \theta_1) + i b \sin(2\pi \theta_1)$ and 
$z_2 = r_2\bigl(\cos(2\pi \theta_2) + i b \sin(2\pi \theta_2)\bigr)$, with $(\theta_1,\theta_2,r_2) \in [0,1]^3$. This yields
\[
\int_{\partial D_1 \times D_2} \hat{V}^H z_2^j U_1 \hat{W}
=
4\pi^2
\int_{[0,1]^3}
\hat{V}^H z_2^j U_1 \hat{W} \;
b r_2 \,\psi(\theta_1)\, d r_2 \, d\theta_1 \, d\theta_2,  \, j=0,\ldots,2M-1,
\]
where
$
\psi(\theta_1) = -\sin(2\pi \theta_1) + i b \cos(2\pi \theta_1)
$
arises from the differential $dz_1$.   We discretize this integral using a tensor-product quadrature rule:
\begin{itemize}
\item trapezoidal rules with $N_1$ and $N_2$ points for $\theta_1$ and $\theta_2$,
\item Gauss--Legendre quadrature with $N_r$ nodes $\rho_t$ and weights $w_t$ for $r_2$.
\end{itemize}
This gives nodes $ \theta_{1,\ell} = \frac{\ell}{N_1}$,  $\theta_{2,n} = \frac{n}{N_2}$,  $\zeta_{1,\ell} = \cos(2\pi \theta_{1,\ell}) + i b \sin(2\pi \theta_{1,\ell})$, and 
$
\zeta_{2,t,n} = \rho_t \bigl(\cos(2\pi \theta_{2,n}) + i b \sin(2\pi \theta_{2,n})\bigr).
$
The resulting approximation is
\begin{equation} \label{eq:quadrature}
\int_{\partial D_1 \times D_2} \hat{V}^H z_2^j U_1 \hat{W}
\approx
\frac{4\pi^2 b}{N_1 N_2}
\sum_{\ell=1}^{N_1}
\sum_{n=1}^{N_2}
\sum_{t=1}^{N_r}
w_t \rho_t \,
\hat{V}^H (\zeta_{2,t,n})^j U_1 \hat{W} \;
\psi(\theta_{1,\ell}).
\end{equation}
An analogous construction applies to $U_2$, yielding a complete quadrature rule for the residue integral.

In practice, we first evaluate $\hat{V}^H U_i \hat{W}$ at each quadrature point, and then reuse these values to compute all moments by multiplying with $(z_2)^j$. This makes the marginal cost of increasing the number of moments negligible, while increasing the sketch size (i.e., the number of columns of $\hat V,\hat W$) remains comparatively expensive.

\subsubsection{Complex vs.\ Real Solutions}
\label{subsubsec:real}

The use of a polyellipse, rather than a polydisc, is motivated by the need to isolate real solutions. Many applications are concerned only with real eigenvalues, and it is therefore desirable to concentrate the contour near $\mathbb{R}^2$.  Although $\mathbb{R}^2$ is not an open subset of $\C^2$, a polyellipse with small imaginary axis provides an effective region.  When $b=1$, the region reduces to a polydisc.  For real-valued problems, we instead take $b \ll 1$ (typically $b=10^{-6}$), producing a thin neighborhood of the real box.  Empirically, this significantly improves performance in problems where many complex solutions are present but only a small number of real solutions are of interest (see \cref{subsec:arma}).

\section{Analysis}
\label{sec:analysis}

We give a brief complexity analysis of multiBeyn, with a more explicit discussion of the two-dimensional implementation biBeyn, and then analyze the conditioning of the reduced eigenvalue problem.

\subsection{Complexity of multiBeyn}
\label{subsec:comp}

The central goal of complexity analysis is to show significant gains over non-contour solvers in both space and time complexity. We summarize this result in \cref{tab:comp}.  To simplify,  we assume (1) all $P_i$ have the same size $n\times n$, (2) the contour is a polyellipse $D$ in $d$ variables, and (3) the reduced eigenproblem is negligible because the number of enclosed eigenvalues is small.  Therefore, the complexity of multiBeyn is dominated by the numerical evaluation of the contour integral. As demonstrated in~\cite{kravanja1998contour,kravanja2000contour}, and generalizing our \cref{subsec:quad}, integrating on $D$ requires $2d-1$ variables; we assume that we use $N_p$ points in each one of those variables (in the language of \cref{subsec:quad}, $N_1=N_2=N_r=N_p$). Therefore the overall cost is the cost of one quadrature point times $N_p^{2d-1}$.

\begin{table}[]
    \centering
    \caption{Complexity of multiBeyn versus operator determinants. The complexity estimate for multiBeyn assumes $m \ll n^d$ so that the reduced eigenvalue problem is negligible.}
    \setlength{\tabcolsep}{5.5pt}
    \begin{tabular}{ccccc}
    \hline \\
     & multiBeyn & Operator Det & biBeyn & Operator Det 2D \\
     Time Complexity & $O\left(k^2 N_p^{2d-1} d^2 n^{d+1}2^d\right)$ & $\mathcal{O}(n^{3d})$ & $\mathcal{O}(k^2 N_p^3 n^3)$ & $\mathcal{O}(n^6)$ \\
     Space Complexity & $\mathcal{O}(k n^d)$ & $\mathcal{O}(n^{2d})$ & $\mathcal{O}(k n^2)$ & $\mathcal{O}(n^4)$ \\
    \hline \\
    \end{tabular}
    \label{tab:comp}
\end{table}

At each quadrature point we must evaluate a local representative of $\eta_\omega$, which is a sum of terms involving up to $d-1$ factors of $\bar\partial \rho_i$ multiplied by $\omega$, i.e., 
$$
\vv^H \sum_{\sigma \in S_{d-1}} (-1)^{\text{sgn}(\sigma)}  \overline{\partial} \rho_{\sigma(1)} \wedge \cdots \wedge \widehat{\overline{\partial} \rho_{\sigma(\ell)}} \wedge \cdots \wedge \overline{\partial} \rho_{\sigma(d)} \wedge \omega \ww,
$$
for $k$ vectors $\vv$ and $\ww$, where $\overline{\partial} \rho_{\sigma(i)}$ can be expanded as in \cref{eq:dbrho}.  In the general $d$-dimensional case, and without the simplifications available for $d=2$,  we multiply $\ww$ through from right to left. To calculate terms $\left(\sum_{j=1}^d Q_j Q_j^H\right) \ww$ for arbitrary $\ww$, we can apply the same method as in \cref{subsec:sylv}, so that there is an initial cost of Schur decomposition of matrices of size $n$, which is $\mathcal{O}(n^3)$, plus a cost of $dn^{d+1}$ for the multiplications required to rotate $\ww$ into and out of the diagonalizing basis. Next, we have multiplications of the form $Q_i Q_i^H \ww$, which has cost $\mathcal{O}(n^{d+1})$ when we reshape $\ww$ to a tensor and only multiply in the $i$-th mode. The product $Q_i d Q_i^H \ww$ has cost $\mathcal{O}(d n^{d+1})$, as it contains $d$ copies of the same reshaped tensor product. Multiplying $\omega \ww$ costs $\mathcal{O}(d n^2)$, because we can assume $\ww = \otimes_{i=1}^d \ww_i$, for $\ww_i$ each of length $n$, which makes this cost negligible overall.

If multiplying from right to left, after passing through $\overline{\partial} \rho_{\sigma(d)}$, we have $d$ distinct $(d,1)$-forms. After passing through $i$ forms $\overline{\partial} \rho_{\sigma(j)}$, we have $\binom{d}{i}$ distinct $(d,i)$-forms. Naively, the next step is calculated separately for each one of these matrix-valued forms. Putting it all together, the cost of step $i+1$, for a single $(i,d)$ form, working right to left in \cref{eq:dbrho} is
\begin{enumerate}
\item $\mathcal{O}(d n^{d+1})$ for the right-hand product $\ww^{(1)} = \left(\sum_{j=1}^d Q_j Q_j^H\right)^{-1} \ww$.
\item $\mathcal{O}(d n^{d+1})$ for $\ww^{(2)} = \left(\sum_{j=1}^d Q_j (dQ_j)^H\right) \ww^{(1)}$.
\item $\mathcal{O}(d^2 n^{d+1})$ for $\ww^{(3)} = \left(\sum_{j=1}^d Q_j (dQ_j)^H\right) \ww^{(2)}$.
\item $\mathcal{O}(d n^{d+1})$ for $\ww^{(4)} = Q_i Q_i^H \ww^{(3)}$.
\item $\mathcal{O}(n^{d+1})$ for $\ww^{(5)} = Q_i (dQ_i^H) \ww^{(1)}$.
\end{enumerate}
Thus the total cost for this step is $\mathcal{O}(\binom{d}{i} d^2n^{d+1})$, so overall the cost is
$$
\mathcal{O}\left(k N_p^{2d-1} \sum_{i=0}^{d-2} \binom{d}{i} d^2n^{d+1}\right) = 
\mathcal{O}\left(k N_p^{2d-1} (d^2 n^{d+1}2^d)\right).
$$
Then there is a cost of $\mathcal{O}(k^2 d n^d)$ from $k^2 d$ dot products to multiply the vectors $\eta_{\omega} \ww$ by columns of $\hat{V}$, for $d$ separate $(d,d-1)$-forms, and assemble into the final matrix; combining these yields the formula in \cref{tab:comp}.
This remains asymptotically favorable for large $n$ compared with the dense operator-determinant approach, whose cost is $\mathcal{O}(n^{3d})$~\cite{atkinson1972multieig}. 

We now specialize to biBeyn and the three operations described in~\cref{subsec:sylv}. The cost of step 1 is $\mathcal{O}(n^2)$, and the cost of step 3 is $\mathcal{O}(n^3)$, being dominated by Schur decompositions. The cost of step 2 is also $\mathcal{O}(n^3)$, as it is dominated by matrix products of $n \times n$ matrices. The assembly step has cost $\mathcal{O}(k^2 n^2)$. Therefore the overall cost is $\mathcal{O}(k^2 N_p^3 n^3)$, in line with our overall analysis. While asymptotically this remains the same, our simplifications have the effect of reducing the implied constant dramatically in practice. An efficient $d$-dimensional algorithm would require a similar simplification to avoid as many $\mathcal{O}(n^{d+1})$ operations as possible.

An added bonus of any contour algorithm is that, because the quadrature points are independent, the contour computation is embarrassingly parallel. In an ideal parallel wall-clock model with enough cores, the factor $N_p^{2d-1}$ can be removed from the runtime estimate.

The storage cost of multiBeyn is dominated by constructing $\hat{V}^H z_d^j \eta_{\omega} \hat{W}$. We never explicitly construct the large matrix $\eta_{\omega}$, and we have $\hat{V}, \hat{W} \in \C^{N \times k}$, and $\hat{V}^H z_d^j \eta_{\omega} \hat{W} \in \C^{k \times k}$. The storage cost is dominated by storage of $k \times N$ matrices at various points, and is therefore $\mathcal{O}(k n^d)$, because we do not need to store the result at all quadrature points at once, but rather can accumulate as we go. This again compares favorably to the operator determinants method, which has storage cost $\mathcal{O}(n^{2d})$. For biBeyn, we have storage cost $\mathcal{O}(k n^2)$ as opposed to $\mathcal{O}(n^4)$ for operator determinants.

Due to a potentially large constant hidden by the big-$O$ notation, and large values of $N_p$, we have been unable to see the asymptotic improvement in time complexity in our practical experiments. We have, however, seen the crossover in space complexity, in a particularly dramatic fashion. In \cref{subsec:arma,subsec:mpsl}, we demonstrate experiments where the storage cost caused the operator determinants method to exceed the available compute memory, but biBeyn successfully solves the problem.

\subsection{Conditioning Analysis of multiBeyn}
\label{subsec:cond}

It is a natural question to ask whether we can generalize the idea of pseudospectral inclusion from univariate contour solvers~\cite{colbrook2025infbeyn} to multiBeyn. Unfortunately, it is possible for the pseudospectra of MEPs to be unbounded, and also difficult to compare pseudospectra between a $d$-dimensional MEP and a one-dimensional projected GEP, so we view a full pseudospectral inclusion result as difficult to obtain. Instead, we analyze the conditioning of the eigenvalue problem constructed in multiBeyn, which can be seen as pseudospectral inclusion for arbitrarily small perturbations. We first need to define a suitable condition number for MEPs. Suppose that $\zz^* = (z_1^*,\ldots,z_d^*)$ is a simple eigenvalue of the MEP $P = \{P_i,1 \leq i \leq d\}$ in \cref{eq:mepform} in some compact region $C$, with right eigenvectors $\vv_1,\ldots,\vv_d$ and left eigenvectors $\ww_1,\ldots,\ww_d$, normalized so that $||\vv_i||_2 = ||\ww_i||_2 = 1$ for $i = 1,\ldots,d$. Analogously to~\cite{hochstenbach2003conditioning}, we define
$$
B_0= \begin{bmatrix}
    \ww_1^H \frac{\partial P_1}{\partial z_1}(\zz^*) \vv_1 & \cdots &\ww_1^H \frac{\partial P_1}{\partial z_d}(\zz^*) \vv_1 \\
    \vdots & \ddots & \vdots \\
    \ww_d^H \frac{\partial P_d}{\partial z_1} (\zz^*)\vv_d & \cdots &\ww_d^H \frac{\partial P_d}{\partial z_d}(\zz^*) \vv_d \\
\end{bmatrix}.
$$
A normwise condition number of $\zz^*$ can be defined by
$$
\kappa(\zz^*,P)
:=\limsup_{\varepsilon\to0}
\left\{
\frac{\|\Delta\zz^*\|}{\varepsilon}:
(P_i+\Delta P_i)(\zz^*+\Delta\zz)(\vv_i+\Delta\vv_i)=\!0\ \forall i,
\|\Delta P\| \le \varepsilon
\right\}.
$$
where $\Delta P_i \in \Omega(U,\C^{n_i \times n_i})$, $||P_i|| := \max_{\zz \in C} ||P_i(\zz)||$, where the norm on the right-hand side is the operator $2$-norm, or the largest singular value, and $||P||$ is the $2$-norm of the vector $||P_1||,\ldots,||P_d||$. We can relate this definition to $B_0$, similarly to~\cite[Thm.~6]{hochstenbach2003conditioning}.\footnote{In~\cite{hochstenbach2003conditioning}, the authors use a weighted $2$-norm, so our result, with the standard $2$-norm, is slightly different, even for linear MEPs.}
\begin{theorem}
The condition number $\kappa(\zz^*, P)$ is given by
\begin{equation} \label{eq:condnum}
\kappa(\zz^*, P) = \| B_0^{-1} \|.
\end{equation}
\end{theorem}

\begin{proof}
If we expand the equality constraints and keep only the first-order
terms, then we get
$$
\Delta P_i(\zz^*) \vv_i
+ \sum_{j=1}^{d} \Delta z_j
\left(\frac{\partial P_i}{\partial z_j}\right) \vv_i
+ P_i(\zz^*)\,\Delta \vv_i
= \mathcal{O}(\varepsilon^2), \quad 1 \leq i \leq d.
$$
Premultiplying by $\ww_i^{H}$ yields
$$
\ww_i^{H}\,\Delta P_i(\zz^*)\vv_i
+ \ww_i^{H} \sum_{j=1}^{d}
\Delta z_j
\left(\frac{\partial P_i}{\partial z_j}\right)\vv_i
= \mathcal{O}(\varepsilon^2), \quad 1 \leq i \leq d.
$$
By rearranging the equations we obtain the linear system
$$
\begin{bmatrix}
    \ww_1^H \frac{\partial P_1}{\partial z_1}(\zz^*) \vv_1 & \cdots &\ww_1^H \frac{\partial P_1}{\partial z_d}(\zz^*) \vv_1 \\
    \vdots & \ddots & \vdots \\
    \ww_d^H \frac{\partial P_d}{\partial z_1} (\zz^*)\vv_d & \cdots &\ww_d^H \frac{\partial P_d}{\partial z_d}(\zz^*) \vv_d \\
\end{bmatrix}
\begin{bmatrix}
\Delta z_1 \\
\vdots \\
\Delta z_d
\end{bmatrix}
=
-\begin{bmatrix}
\ww_1^{H}\,\Delta P_1(\zz^*) \vv_1 \\
\vdots \\
\ww_d^{H}\,\Delta P_d(\zz^*) \vv_d
\end{bmatrix}
+ \mathcal{O}(\varepsilon^2),
$$
or in shorter form
$$
B_0\, \Delta \zz =
-\begin{bmatrix}
\ww_1^{H} \Delta P_1(\zz^*)\vv_1 \\
\vdots \\
\ww_d^{H} \Delta P_d(\zz^*)\vv_d
\end{bmatrix}
+ \mathcal{O}(\varepsilon^2).
$$

Since $\zz^*$ is an algebraically simple eigenvalue, it follows that
$B_0$ is nonsingular. Thus,
$$
\Delta \zz
= -B_0^{-1}
\begin{bmatrix}
\ww_1^{H} \Delta P_1(\zz^*)\vv_1 \\
\vdots \\
\ww_d^{H} \Delta P_d(\zz^*)\vv_d
\end{bmatrix}
+ \mathcal{O}(\varepsilon^2),
$$
and we conclude
$$
\|\Delta \zz\|
\le \|B_0^{-1}\| \, \varepsilon
+ \mathcal{O}(\varepsilon^2).
$$
Hence, the expression in \cref{eq:condnum} is an upper bound for the condition number.  
To show that this bound can be attained we consider the matrices
$$
\Delta P_i = \varepsilon y_i \ww_i \vv_i^{H},
$$
for $i = 1,\ldots,d$, where $y_i$ is the $i$-th component of the singular vector of $B_0^{-1}$ corresponding to the largest singular value. Notice that $||\Delta P_i|| = \varepsilon |y_i|$, so $||\Delta P|| = \varepsilon$, and
$$
\Delta \zz
= -B_0^{-1}
\begin{bmatrix}
\ww_1^{H} \Delta P_1(\zz^*)\vv_1 \\
\vdots \\
\ww_d^{H} \Delta P_d(\zz^*)\vv_d
\end{bmatrix}
+ \mathcal{O}(\varepsilon^2) = -B_0^{-1}
\varepsilon \yy
+ \mathcal{O}(\varepsilon^2),
$$
so $||\Delta \zz|| = ||B_0^{-1}||\varepsilon + \mathcal{O}(\varepsilon^2)$, and the result follows by taking $\varepsilon \to 0$.
\end{proof}

Now we compare to the condition number of the corresponding eigenvalue $z_d^*$ in \cref{eq:finalevp}. Let $\mathcal{V},\mathcal{W},V,W,V_0,W_0,\Sigma_0,\Lambda$ be as defined in \cref{subsec:mom}. The eigenvalue condition number of $z_d^*$, following~\cite{tisseur2000conditioning}, is 
$$
\kappa(z_d^*) = \frac{||\xx||_2||\yy||_2}{|\xx^* \Sigma_0 \yy|}, 
$$
where $\xx$ and $\yy$ are the left and right eigenvectors of \cref{eq:finalevp} associated to $z_d^*$. Given our derivation of \cref{eq:finalevp}, we can simplify. Let $R_V = V_0^H \mathcal{V}$ and $R_W = \mathcal{W} W_0$, where $\mathcal{V}$ and $\mathcal{W}$ are defined with eigenvectors scaled as in \cref{thm:simplematrixgrothres}. Then $R_V R_W = \Sigma_0$, and we can set
$$
\xx^H = e_j^H R_V^{-1}, \quad \yy = R_W^{-1}e_j,
$$
where $z_d^*$ is the $(j,j)$ entry of the diagonal matrix $\Lambda$. This implies that $\xx^H \Sigma_0 \yy = 1$, so 
$$
\kappa(z_d^*)  = ||R_V^{-1} e_j||_2 ||R_W^{-1}e_j||_2 = ||(V_0^H \mathcal{V})^{-1} e_j||_2 || (\mathcal{W} W_0)^{-1} e_j ||_2.
$$
This formula expresses the full dependence of our final eigenvalue problem on the Hankel matrix construction, which is known to be ill-conditioned for too many moments. To gain additional insight, we suppose that we only calculate two moments, so that
$$
\mathcal{V} = \hat{V}^H V, \quad \mathcal{W} = W^H \hat{W}.
$$
Then
\begin{align*}
\kappa(z_d^*)  &=  ||(V_0^H \hat{V}^H V)^{-1} e_j||_2 || (W^H \hat{W} W_0)^{-1} e_j ||_2 \\
&\leq \frac{1}{\sigma_{\min}(V_0^H \hat{V}^H V) \sigma_{\min}(W^H \hat{W} W_0)} \\
&\leq \frac{1}{\sigma_{\min}(\hat{V}^H V) \sigma_{\min}(W^H \hat{W} )} \\
&\leq \frac{1}{\sigma_{\min}(V) \sigma_{\min}(W) \cos(\theta_V) \cos(\theta_W)},
\end{align*}
where $\theta_V$ and $\theta_W$ are the largest principal angles between $V$ and $\hat{V}$, and $W$ and $\hat{W}$~\cite[p. 329]{golub2013matrix}.
This demonstrates that the conditioning in this case only depends on how well we approximate the subspaces $V$ and $W$, and the intrinsic conditioning of the eigenvector matrices $V$ and $W$, which also depends on the scaling from \cref{thm:simplematrixgrothres}. If we assume that with high probability the principal angles are small, then the conditioning only depends on the intrinsic conditioning of the eigenvector matrices. Unfortunately, it is quite clear that these may be ill-conditioned even for initially well-conditioned eigenvalues. This is similar to the univariate result of~\cite[Thm.~3.1]{colbrook2025infbeyn}, where the pseudospectral inclusion depends on the singular values of the matrices of eigenvectors. Despite the fact that this is not a general guarantee of a well-conditioned final eigenvalue problem, our formulas allow for quick analysis of stability on a case-by-case basis, as univariate results do.

\section{Numerical Experiments}
\label{sec:num}

We test biBeyn on practical examples from delay-differential equations, multiparameter Sturm-Liouville problems, and ARMA models. This demonstrates that biBeyn succeeds in solving previously infeasible real-world problems across many different types of MEPs: linear, nonlinear polynomial, and nonpolynomial analytic. We performed all of the experiments on an AMD Ryzen 9 5950x using MATLAB r2023a. Code to reproduce all of the experiments can be found in~\cite{graf2026beyn}.\footnote{MATLAB's parfor is not exactly deterministic due to variations in the order of operations coming from different assignments to workers during different runs of an experiment, so the experiments are not precisely reproducible, but are as close as possible given this restriction.}

\subsection{Delay-Differential Equations}
\label{subsec:dde}

Delay-differential equations (DDEs) occur in a wide variety of modeling applications~\cite{kuang1993delay,balachandran2009delay}. An important aspect is the stability of solutions. One approach is to attempt to find the critical delays, which are the delays on the boundary of the region of stable delays. The authors of~\cite{jarlebring2009dde} demonstrate how to find the critical delays for a delay-differential equation with a single delay, or with multiple commensurate delays, meaning that the delays are integer multiples of a single delay, by solving a polynomial MEP. However, a general multiple delay DDE is beyond their methods, because the associated characteristic equation leads to a nonpolynomial MEP. A contour solver is the only possible nonlocal method to solve a nonpolynomial MEP.

Consider the DDE with two delays
\begin{equation} \label{eq:DDE}
B \dot{x}(t) = A_0 x(t) + A_1 x(t - h_1) + A_2 x(t - h_2),
\end{equation}
where $B,A_0,A_1,A_2$ are matrices. The associated characteristic equation is
\begin{equation} \label{eq:DDEcrit}
\left( A_0 + e^{-\lambda h_1} A_1 + e^{-\lambda h_2} A_2 - \lambda B \right) \vv = 0.
\end{equation}
A DDE is stable if all the eigenvalues $\lambda$ are contained in the open left half plane. Therefore, an important step in a stability analysis is to find pairs $(h_1,h_2)$ that give a purely imaginary eigenvalue $\lambda$, where the DDE may cross over from stable to unstable~\cite{jarlebring2009dde}. Because we want purely imaginary $\lambda$, we take the complex conjugate of \cref{eq:DDEcrit}, and obtain a two-parameter analytic eigenvalue problem
\begin{equation} \label{eq:DDEcritevp}
\begin{aligned}
\left( A_0 + e^{-\lambda h_1} A_1 + e^{-\lambda h_2} A_2 - \lambda B \right) \vv_1 = 0, \\
\left( \overline{A}_0 + e^{\lambda h_1} \overline{A}_1 + e^{\lambda h_2} \overline{A}_2 + \lambda \overline{B} \right) \vv_2 = 0,
\end{aligned}
\end{equation}
Previously, the best method to solve such a problem was the local method from~\cite{plestenjak2016mep}, which accurately traces the critical delay curve once a single point on the curve is found. However, we can solve nonpolynomial MEPs directly, so we can plug in arbitrary values of $h_2$ and find the resulting $h_1$ in the critical delay pair $(h_1,h_2)$. To demonstrate the effectiveness of our method, we solve a problem with $2 \times 2$ random matrices in~\cite{graf2026beyn}, with the resulting curve in \cref{fig:DDE}. We sample evenly spaced points in $h_2$ between $h_2=1$ and $h_2 = 3$, and solve for the resulting value of $h_1$ within our search region.

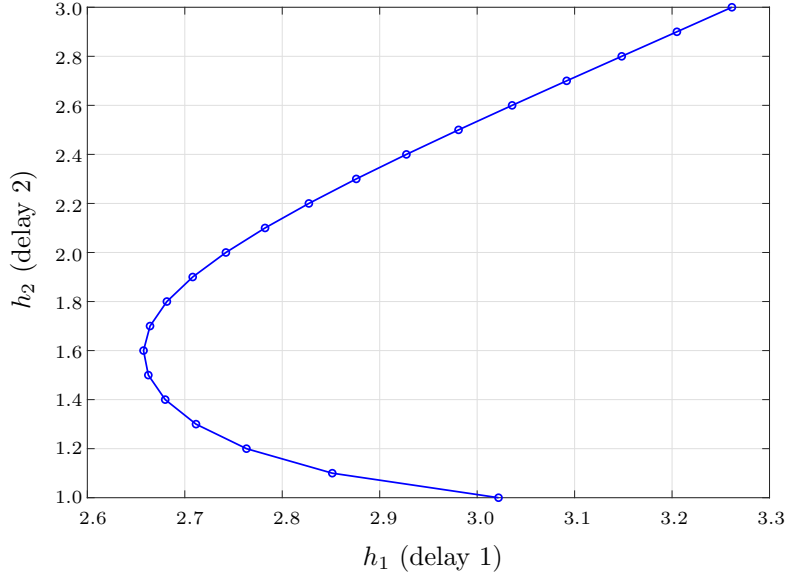
\begin{figure}
\begin{center}
\begin{overpic}[width=.7\textwidth]{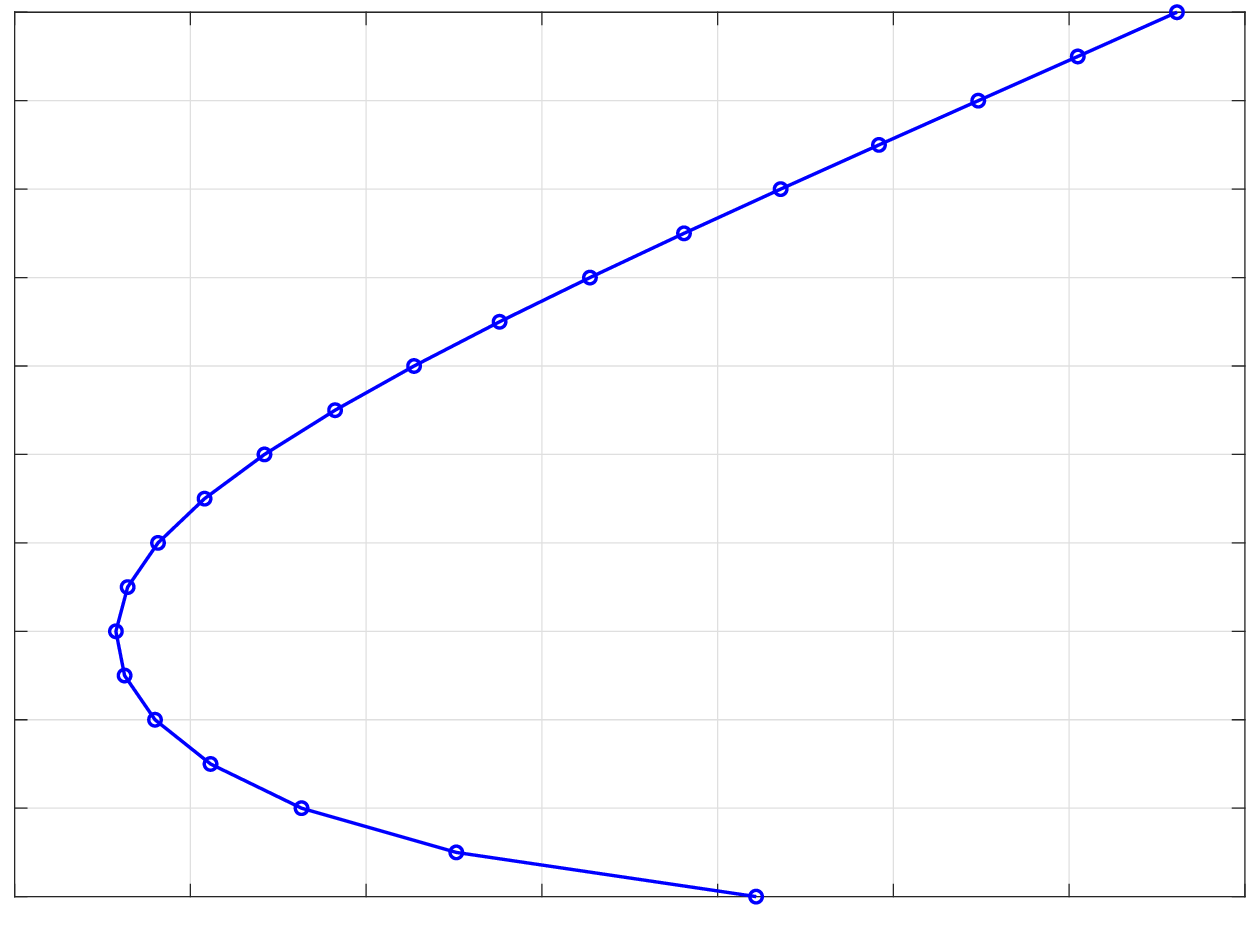}
\put(41,-7){$h_1 \text{ (delay 1)}$}
            \put(-10,30){\rotatebox{90}{$h_2 \text{ (delay 2)}$}}
             {\footnotesize
            \put(-3.5,2){$1.0$}
            \put(-3.5,9){$1.2$}
            \put(-3.5,16){$1.4$}
            \put(-3.5,23){$1.6$}
            \put(-3.5,30){$1.8$}
            \put(-3.5,37){$2.0$}
            \put(-3.5,44.1){$2.2$}
            \put(-3.5,51.2){$2.4$}
            \put(-3.5,58.3){$2.6$}
            \put(-3.5,65.4){$2.8$}
            \put(-3.5,72.5){$3.0$}
            \put(0,-1){$2.6$}
            \put(14,-1){$2.7$}
            \put(28,-1){$2.8$}
            \put(42,-1){$2.9$}
            \put(56,-1){$3.0$}
            \put(70,-1){$3.1$}
            \put(84,-1){$3.2$}
            \put(98,-1){$3.3$}
}
\end{overpic}
\end{center}
\vspace{.5cm}
\caption{Critical delay pairs for the DDE in~\cref{eq:DDE}, found as solutions to~\cref{eq:DDEcrit} and~\cref{eq:DDEcritevp} with purely imaginary $\lambda$.}
\label{fig:DDE}
\end{figure}

\subsection{Multiparameter Sturm-Liouville Problems}
\label{subsec:mpsl}

Sturm-Liouville problems are a common application of univariate contour methods, because they have infinitely many eigenvalues, but there are very accurate asymptotic formulas that allow for a contour method to search only small regions to find eigenvalues in a particular order~\cite{horning2020feast}. A similar theory exists for multiparameter Sturm-Liouville problems~\cite{atkinson2010multieig}. Consider the multiparameter Sturm-Liouville problem on $[0,1]$ given by
\begin{equation}
\label{eq:mpsl}
\begin{aligned}
y_1''(x_1) + \{ \lambda_1 x_1 + \lambda_2(1-x_1) \} y_1(x_1) &= 0, \\
y_2''(x_2) + \{ \lambda_1 + \lambda_2 x_2 \} y_2(x_2) = 0,
\end{aligned}
\end{equation}
with $y_1(0) = y_1(1) = y_2(0) = y_2(1) = 0$. Similarly to how the eigenvalues of a univariate Sturm-Liouville problem can be numbered, we can obtain asymptotics for the eigenvalues from oscillation numbers $n_1,n_2$, as in~\cite[\S~8.6]{atkinson2010multieig}.
The system in~\cref{eq:mpsl} fits into the framework of~\cite[\S~8.6]{atkinson2010multieig}, with coefficient rows
\[
(p_{11}(x_1), p_{12}(x_1)) = (x_1, 1-x_1), \quad (p_{21}(x_2), p_{22}(x_2)) = (1, x_2).
\]
The associated phase--integral functions, as in~\cite[\S~8.6]{atkinson2010multieig}, are
\[
F_r(\lambda_1, \lambda_2) = \int_0^1 \sqrt{p_{r1}(x_r)\lambda_1 + p_{r2}(x_r)\lambda_2}\, dx_r, 
\quad r = 1,2.
\]
By~\cite[Thm. 8.6.2--8.6.3]{atkinson2010multieig}, the eigenvalues $(\lambda_1, \lambda_2)$ associated with oscillation numbers $(n_1,n_2)$ satisfy
\[
F_r(\lambda_1, \lambda_2) = \pi n_r + o(n_1 + n_2), 
\quad r = 1,2, \quad \text{as } n_1^2 + n_2^2 \to \infty.
\]
Elementary integration yields
\begin{align*}
F_1(\lambda_1, \lambda_2) 
&= \int_0^1 \sqrt{\lambda_1 x_1 + \lambda_2(1-x_1)}\, dx_1  =
\frac{2}{3}\,\frac{\lambda_1^{3/2} - \lambda_2^{3/2}}{\lambda_1 - \lambda_2}, \\
F_2(\lambda_1, \lambda_2) 
&= \int_0^1 \sqrt{\lambda_1 + \lambda_2 x_2}\, dx_2 =
\frac{2}{3\lambda_2}\left[(\lambda_1 + \lambda_2)^{3/2} - \lambda_1^{3/2}\right].
\end{align*}
Thus the leading-order asymptotic equations are
\begin{align}
\frac{2}{3}\,\frac{\lambda_1^{3/2} - \lambda_2^{3/2}}{\lambda_1 - \lambda_2} &\sim \pi n_1, \\
\frac{2}{3\lambda_2}\left[(\lambda_1 + \lambda_2)^{3/2} - \lambda_1^{3/2}\right] &\sim \pi n_2.
\end{align}
We can then solve this system of equations by Newton's method for any particular choice of oscillation numbers $n_1,n_2$.
To solve the Sturm-Liouville system and find the true eigenvalue pair, we discretize in a region of each asymptotic estimate using matrices of size $200 \times 200$ to ensure that we have at least two points per wavelength. Then we solve the resulting linear MEP using biBeyn. For comparison, a two-parameter eigenvalue problem with matrices of size $200 \times 200$ results in operator determinants of size $40,000 \times 40,000$, which caused our computer to run out of memory and crash. Therefore, biBeyn is the first and, to our knowledge, the only nonlocal method that can find solutions for such high oscillation numbers. We give the results for some oscillation numbers between $30$ and $100$ in~\cref{tab:mpsl} and plot one of the resulting eigenfunction pairs in~\cref{fig:mpsl}.

\begin{table}[h]
\centering
\begin{tabular}{cc|cc|cc}
$n_1$ & $n_2$ 
& $\lambda_{1,\mathrm{asy}}$ & $\lambda_{2,\mathrm{asy}}$
& $\lambda_{1,\mathrm{num}}$ & $\lambda_{2,\mathrm{num}}$ \\
\hline
30 & 40 & 13498.430479 & 4642.853415 & 13497.825582 & 4644.150768 \\
40 & 50 & 18028.248411 & 13606.167147 & 18028.396333 & 13606.074993 \\
50 & 60 & 22577.169734 & 26842.168609 & 22617.547999 & 26841.798654 \\
60 & 80 & 53993.721914 & 18571.413659 & 53993.115657 & 18572.712393 \\
80 & 100 & 72112.993645 & 54424.668586 & 72113.143417 & 54424.574671 \\
\end{tabular}
\caption{Asymptotic and numerical solutions $(\lambda_1,\lambda_2)$ to the multiparameter Sturm-Liouville problem in~\cref{eq:mpsl}.}
\label{tab:mpsl}
\end{table}

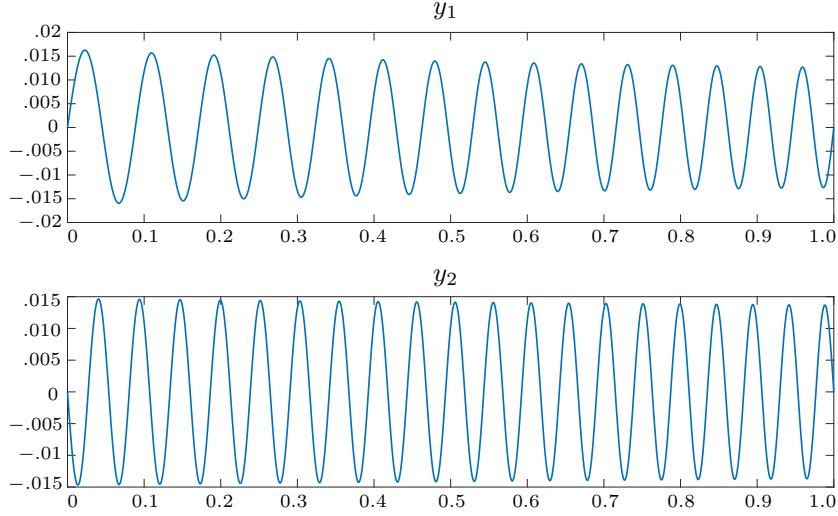
\begin{figure}
\begin{center}
\begin{overpic}[width=\textwidth]{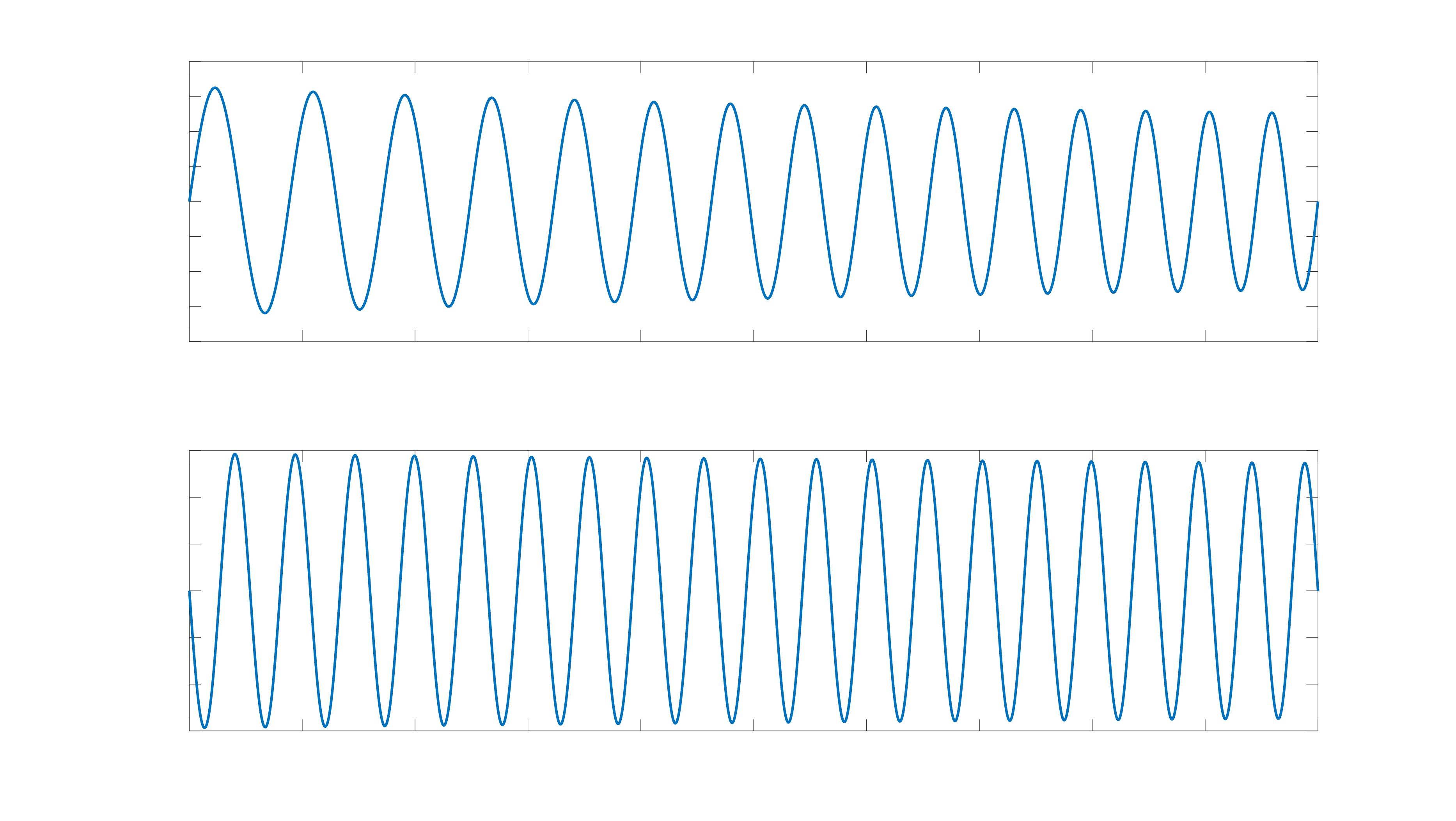}
	   \put(50,54){$y_1$}
           \put(50,27){$y_2$}
{\footnotesize
            \put(8,32.5){$-.02$}
            \put(7,34.9){$-.015$}
            \put(8,37.3){$-.01$}
            \put(7,39.7){$-.005$}
            \put(11,42.1){$0$}
            \put(8.5,44.5){$.005$}
            \put(8.5,46.9){$.010$}
            \put(8.5,49.3){$.015$}
            \put(9.6,51.7){$.02$}

            \put(13,31){$0$}
            \put(19.5,31){$0.1$}
            \put(27,31){$0.2$}
            \put(34.5,31){$0.3$}
            \put(42.5,31){$0.4$}
            \put(50.5,31){$0.5$}
            \put(58,31){$0.6$}
            \put(66,31){$0.7$}
            \put(73.7,31){$0.8$}
            \put(81.5,31){$0.9$}
            \put(88,31){$1.0$}
            
            \put(7,6){$-.015$}
            \put(8,9){$-.01$}
            \put(7,12){$-.005$}
            \put(11,15){$0$}
            \put(8.5,18.5){$.005$}
            \put(8.5,21.5){$.010$}
            \put(8.5,24.5){$.015$}
            
            \put(13,4.2){$0$}
            \put(19.5,4.2){$0.1$}
            \put(27,4.2){$0.2$}
            \put(34.5,4.2){$0.3$}
            \put(42.5,4.2){$0.4$}
            \put(50.5,4.2){$0.5$}
            \put(58,4.2){$0.6$}
            \put(66,4.2){$0.7$}
            \put(73.7,4.2){$0.8$}
            \put(81.5,4.2){$0.9$}
            \put(88,4.2){$1.0$}
}
\end{overpic}
\end{center}
\caption{Solutions $y_1,y_2$ to the multiparameter Sturm-Liouville problem in~\cref{eq:mpsl} for oscillation numbers $30$,  $40$. The method of~\cite{atkinson1972multieig} caused our computer to run out of memory.}
\label{fig:mpsl}
\end{figure}

\subsection{ARMA Model Parameter Selection}
\label{subsec:arma}

An ARMA model is a popular way to model time series with oscillating error terms. In~\cite{vermeersch2019arma}, the authors present a manner in which to compute all critical points in the parameter space of an ARMA model by solving a large polynomial multiparameter eigenvalue problem. The alternative is to compute critical points via a local optimization method, which runs the risk of missing the correct solution entirely. In~\cite{hochstenbach2024numerical}, the authors optimize this idea by taking advantage of the structure of these MEPs to compress them as much as possible, which significantly speeds up the resulting algorithm. However, even this approach has limitations. We focus here on the ARMA(1,1) model, but the main ideas generalize to any size. An ARMA(1,1) model for a time series $\{y_i\}$ takes the form
$$y_i = \alpha_1 y_{i-1} + \gamma_1 e_{i-1},$$
where $e_i$ is an error term. The objective is to find the parameters $\alpha_1,\gamma_1$ that minimize the norm of the error vector $\mathbf{e}$. 
Given a sequence $\mathbf{y} \in \mathbb{R}^n$, we can find the critical points by solving the two-parameter rectangular quadratic eigenvalue problem~\cite{vermeersch2019arma, hochstenbach2024numerical}
\begin{equation}
\left( A_{00} + \alpha A_{10} + \gamma A_{01} + \gamma^2 A_{02} \right) \mathbf{x} = 0
\label{eq:r2ep}
\end{equation}
with matrices $A_{ij}$ of size $(3n-1)\times(3n-2)$, where
\begin{equation}
\begin{aligned}
A_{00} = 
\begin{bmatrix}
\mathbf{y}_{(2)} & I & 0 & 0 \\
\mathbf{y}_{(1)} & 0 & I & 0 \\
0 & R & 0 & I \\
0 & \mathbf{y}_{(1)}^T & \mathbf{y}_{(2)}^T & 0 \\
0 & 0 & 0 & \mathbf{y}_{(2)}^T
\end{bmatrix}, \quad
&A_{10} =
\begin{bmatrix}
\mathbf{y}_{(1)} & 0 & 0 & 0 \\
0 & 0 & 0 & 0 \\
0 & 0 & 0 & 0 \\
0 & 0 & \mathbf{y}_{(1)}^T & 0 \\
0 & 0 & 0 & \mathbf{y}_{(1)}^T
\end{bmatrix}, \\
A_{01} =
\begin{bmatrix}
0 & R & 0 & 0 \\
0 & 0 & R & 0 \\
0 & 2I & 0 & R \\
0 & 0 & 0 & 0 \\
0 & 0 & 0 & 0
\end{bmatrix}, \quad
&A_{02} =
\begin{bmatrix}
0 & I & 0 & 0 \\
0 & 0 & I & 0 \\
0 & 0 & 0 & I \\
0 & 0 & 0 & 0 \\
0 & 0 & 0 & 0
\end{bmatrix},
\end{aligned}
\label{eq:arma11evp}
\end{equation}
where $\mathbf{y}_{(1)}$ is the first $n-1$ entries of $\mathbf{y}$, $\mathbf{y}_{(2)}$ is the last $n-1$ entries of $\mathbf{y}$, and $R$ is a tridiagonal matrix with stencil $[1,0,1]$. In~\cite{hochstenbach2024numerical}, the authors explain how to sketch this to a two-parameter quadratic eigenvalue problem with matrices of size $(3n-1) \times (3n-1)$, with two equations, which is suitable for biBeyn. For this specific problem, they also show that this can be compressed and then reduced through a variant of the operator determinants method to a generalized eigenvalue problem of size $(9n(n-1)+2) \times (9n(n-1)+2)$. However, it is also demonstrated experimentally that there are only $28n+7$ total finite eigenvalues. Furthermore, the only solutions of interest are real solutions. Our experimental evidence and that of~\cite{hochstenbach2024numerical} suggests that there are generally $3$ real solutions, regardless of the length of the sequence $\yy$. 

Therefore, the time complexity of the method of~\cite{hochstenbach2024numerical} is $\mathcal{O}(n^6)$ and the storage cost is $\mathcal{O}(n^4)$. For biBeyn, this drops to $\mathcal{O}(n^3)$ and $\mathcal{O}(n^2)$, with all the cost in the calculation of the contour integral. In particular, for this problem, our innovative choice of polyellipse contours centered on the real plane is crucial. For large problems (see~\cref{fig:ARMA2}), there are many complex solutions with small real parts. These solutions can cause biBeyn to fail to find all real solutions when using a polydisc contour. With a polyellipse contour we are able to push past previous numerical results to compute the critical points for extremely long sequences. First, we confirm that we can match~\cite{hochstenbach2024numerical} on a short sequence from~\cite[ex.~7]{hochstenbach2024numerical} given by
$$
\begin{matrix} y = &
[2.4130 & 1.0033 & 1.2378 & -0.72191& -0.81745 & -2.2918  &\\
& & 0.18213 & 0.073557 & 0.55248 & 2.0180 & 2.6593 & 1.1791].
\end{matrix}
$$
We show the critical points computed by each method in~\cref{fig:ARMA1}, which are identical.
\begin{figure}
\begin{center}
\begin{overpic}[width=.8\linewidth]{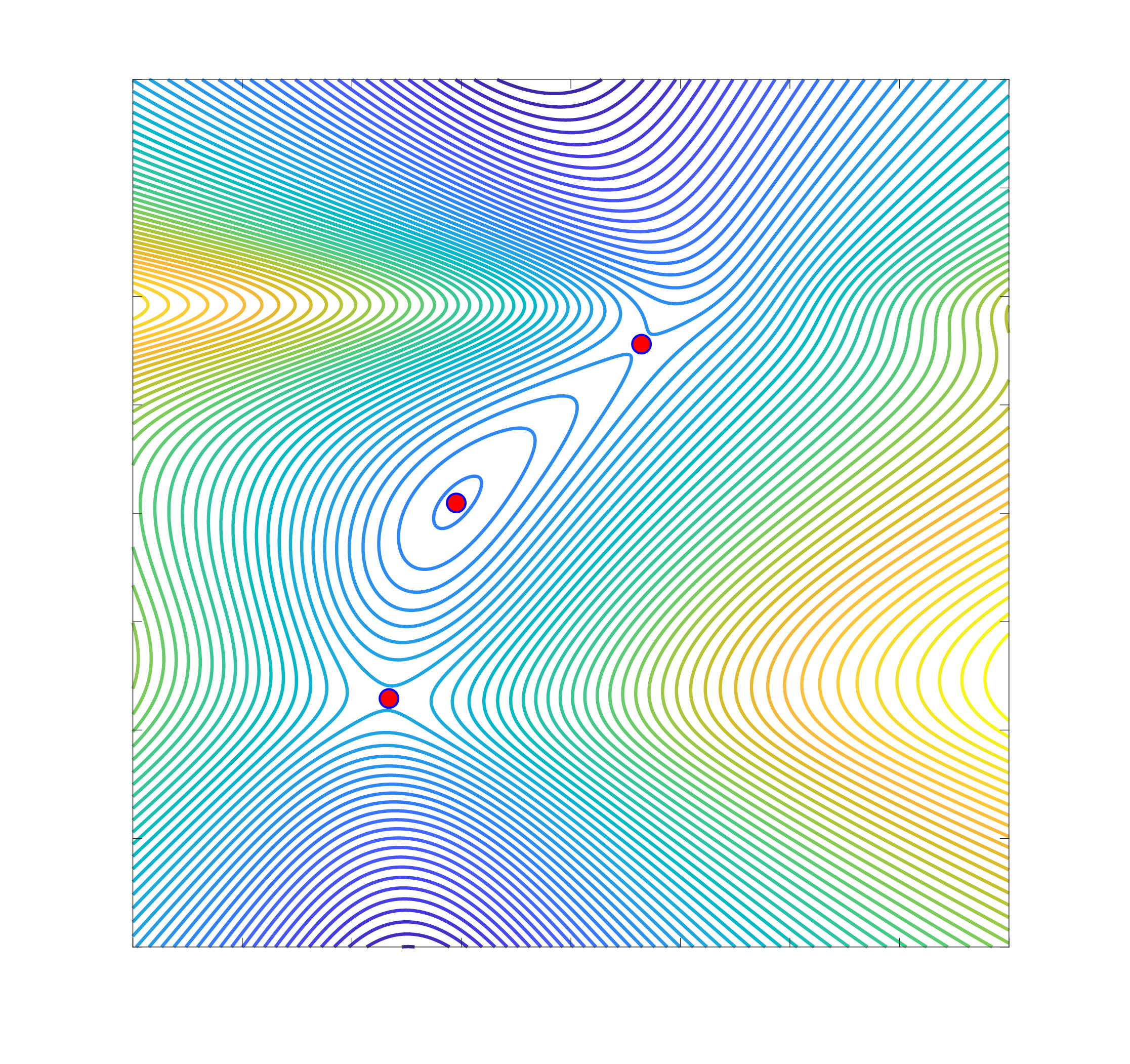} 
 	\put(50,3){$\alpha_1$}
            \put(1,46){\rotatebox{90}{$\gamma_1$}}
             {\footnotesize
            \put(5.5,10.5){$-2.0$}
            \put(5.5,19.5){$-1.5$}
            \put(5.5,28.8){$-1.0$}
            \put(5.5,38.4){$-0.5$}
            \put(10,48){$0$}
            \put(7.7,58){$0.5$}
            \put(7.7,67.5){$1.0$}
            \put(7.7,77){$1.5$}
            \put(7.7,86){$2.0$}
            \put(10,8){$-2.0$}
            \put(18.3,8){$-1.5$}
            \put(28,8){$-1.0$}
            \put(38,8){$-0.5$}
            \put(50,8){$0$}
            \put(58.5,8){$0.5$}
            \put(68,8){$1.0$}
            \put(78,8){$1.5$}
            \put(88,8){$2.0$}
           }
\end{overpic}
\end{center}
\caption{Critical points of an ARMA model found by~\cite{hochstenbach2024numerical} (blue circles) and by biBeyn~\cite{graf2026beyn} (red dots) on a sequence of length $12$. The methods return the same set of critical points.}
\label{fig:ARMA1}
\end{figure}
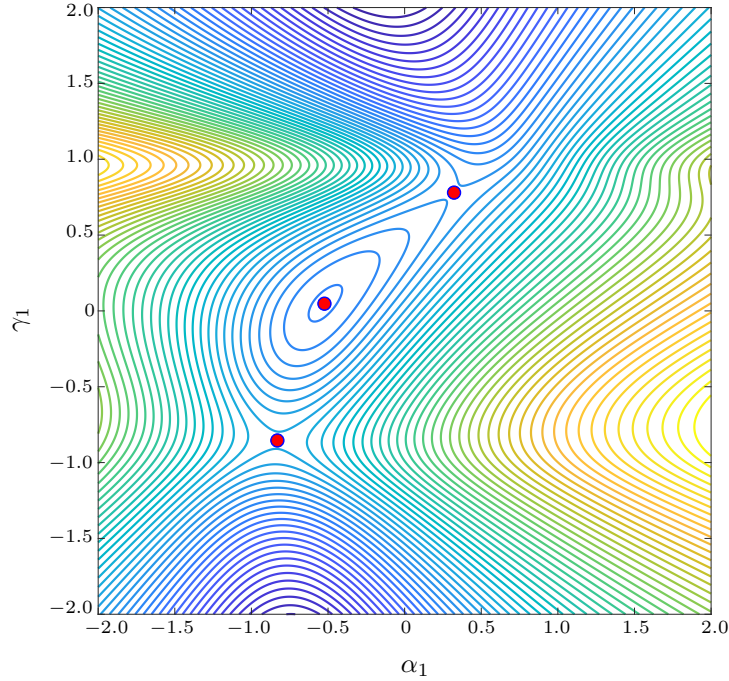

In practice, it is useful to be able to model time series that are far longer than $12$ time steps. We test on practical data from the crack propagation benchmark in~\cite{langlois2018fatigue}. We take a sequence of length $64$ from this data, and try to calculate the optimal model parameters. If using the method of~\cite{hochstenbach2024numerical}, we must solve a generalized eigenvalue problem of size $36,290 \times 36,290$, which caused our computer to run out of memory. However, biBeyn can find all the critical points, as shown in~\cref{fig:ARMA2}. 

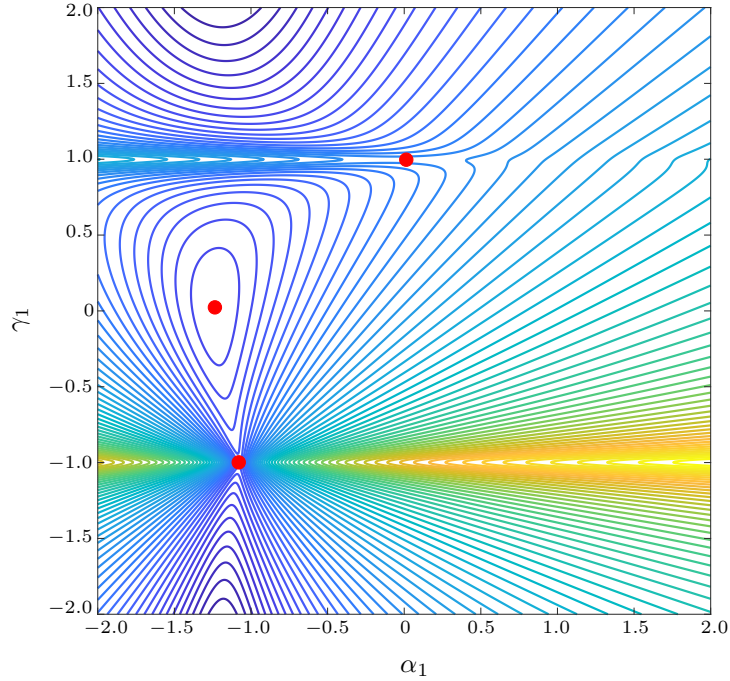
\begin{figure}
\begin{center}
\begin{overpic}[width=.8\linewidth]{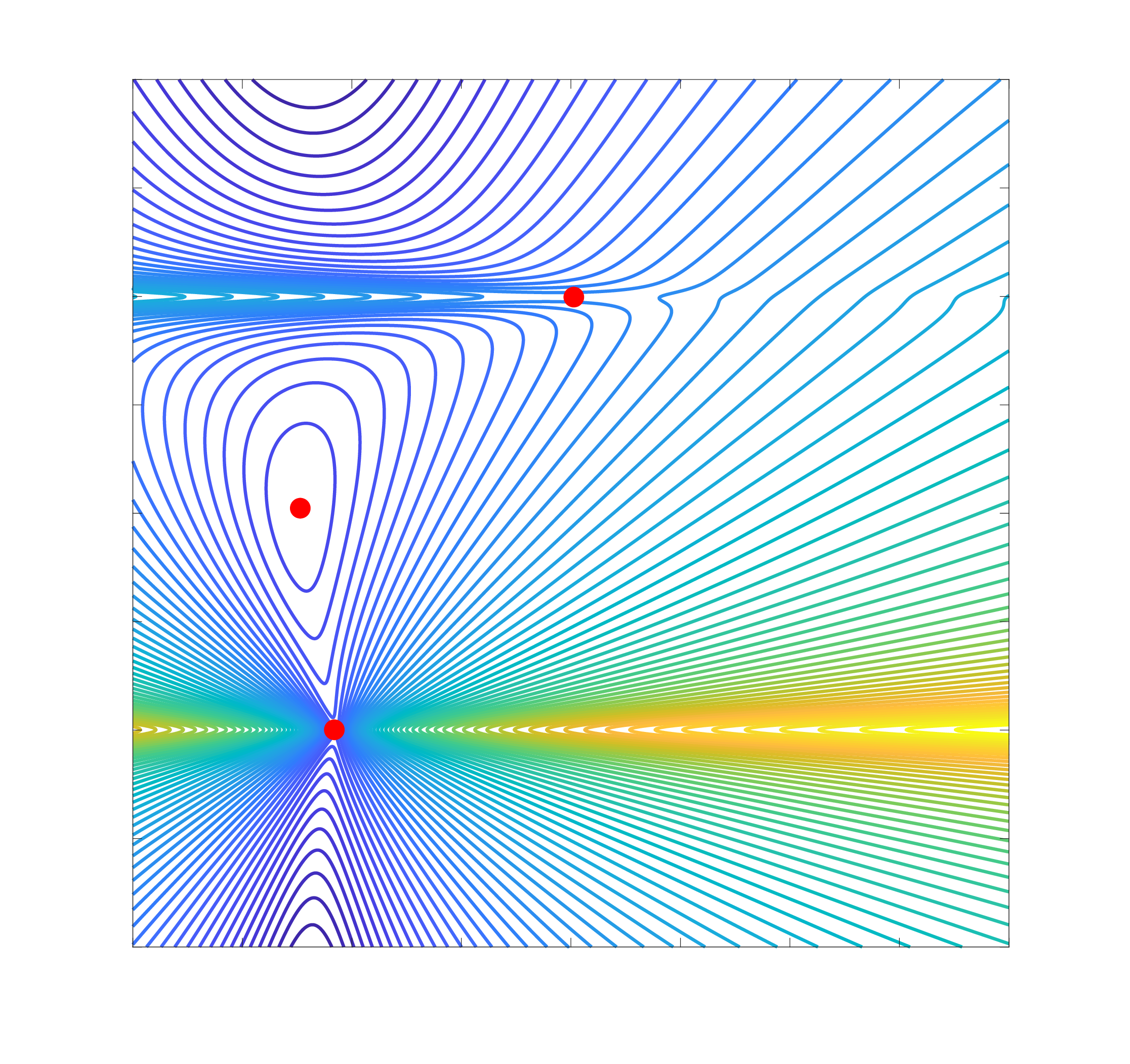} 
 	\put(50,3){$\alpha_1$}
            \put(1,46){\rotatebox{90}{$\gamma_1$}}
             {\footnotesize
            \put(5.5,10.5){$-2.0$}
            \put(5.5,19.5){$-1.5$}
            \put(5.5,28.8){$-1.0$}
            \put(5.5,38.4){$-0.5$}
            \put(10,48){$0$}
            \put(7.7,58){$0.5$}
            \put(7.7,67.5){$1.0$}
            \put(7.7,77){$1.5$}
            \put(7.7,86){$2.0$}
            \put(10,8){$-2.0$}
            \put(18.3,8){$-1.5$}
            \put(28,8){$-1.0$}
            \put(38,8){$-0.5$}
            \put(50,8){$0$}
            \put(58.5,8){$0.5$}
            \put(68,8){$1.0$}
            \put(78,8){$1.5$}
            \put(88,8){$2.0$}
           }
\end{overpic}
\end{center}
\caption{Critical points of an ARMA model found by biBeyn~\cite{graf2026beyn} on a sequence of length $64$. The method from~\cite{hochstenbach2024numerical} caused our machine to run out of memory.}
\label{fig:ARMA2}
\end{figure}

\subsection{A Toy Trivariate Test}
\label{subsec:triBeyn}

To test the validity of our general formulas in \cref{sec:theory}, we implemented a rudimentary trivariate version of our method as well. We started by asking a large language model to generalize our simplified formulas in \cref{sec:details}, and then to implement its generalized trivariate formula. While this approach was ultimately unsuccessful, the LLM created impressively plausible 3D formulas and code. We were inspired by the relative simplicity of the LLM's work to derive our own formulas and implement them in~\cite{graf2026beyn}. In the end, the formulas we use are derived by hand, but we view this as an intriguing probe of the limits of current LLM coding capabilities for new software.

A trivariate method requires the computation of the contour integral
$$
\int_{\partial D} g(\zz)(U_1 + U_2 + U_3) d\zz,
$$
where $D$ is the unit polydisc in $\C^3$, and
\begin{align*}
    U_1 &= R_{23} - R_{32} \\
    U_2 &= -(R_{13} - R_{31}) \\
    U_3 &= R_{12} - R_{21}, \\
\end{align*}
with 
$$
R_{rs} = \overline{\partial}\tilde{\rho}_{i,ijk,r}^H \overline{\partial}\tilde{\rho}_{j,jk,s}^H Q_k^{-1} -\overline{\partial}\tilde{\rho}_{j,ijk,r}^H \overline{\partial}\tilde{\rho}_{i,ik,s}^H Q_k^{-1},
$$
for any assignment of $\{i,j,k\} = \{1,2,3\}$ (all are equivalent), where
\begin{align*}
\overline{\partial}\tilde{\rho}_{i,I \cup \{i\},\ell}^H &= (-1)^{i-1} \left(\frac{\partial Q_i^H}{\partial z_{\ell}} - (Q_i^H)D_{(I)}^{-1}S_{(I),\ell} \right) D_{(I)}^{-1} d \overline{z}_{\ell}, \\
S_{(I),\ell} &= \sum_{j \notin I} Q_j \frac{\partial Q_j^H}{\partial z_{\ell}} + \sum_{j \in I} \frac{\partial Q_j^H}{\partial z_{\ell}} Q_j, \\
D_{(I)} &= \sum_{j\notin I} Q_j Q_j^H + \sum_{j \in I} Q_j^H Q_j.
\end{align*}

As in the bivariate case, $U_i$ vanishes except for on one particular piece of the boundary of the unit polydisc, so we can split the contour integral into three pieces for $U_1,U_2,U_3$, and integrate on each piece using the three dimensional extension of the quadrature formula in \cref{eq:quadrature}. After we compute the potential third coordinates of solutions, we plug them back in one at a time and use a version of biBeyn to compute the second coordinates, and then univariate Beyn's method for the first coordinates. We did not optimize the method as we did for biBeyn, so we do not expect it to perform exceptionally well in practice, but it is sufficient to test our theory.
To test our trivariate implementation, we solved the following generic toy example:
\begin{align*}
\left(\begin{bmatrix} -0.65 & -0.76 \\ 1.18 & -1.11 \end{bmatrix} \!+\! z_1 \begin{bmatrix} -1.69 & -1.12 \\ -1.15 & 0.36 \end{bmatrix} \!+\! z_2 \begin{bmatrix} -0.39 & -1.70 \\ 1.17 & 1.60 \end{bmatrix} \!+\! z_3 \begin{bmatrix} -3.02 & -0.49 \\ 1.75 & 0.33 \end{bmatrix}\right) \vv_1 &\!= 0, \\
\left( \begin{bmatrix} -1.97 & 1.18 \\ -1.27 & 2.03 \end{bmatrix} \, + \, z_1 \, \begin{bmatrix} -0.55 & 3.56 \\ 1.21 & 3.55 \end{bmatrix} \, + \, z_2 \,\begin{bmatrix} -3.73 & -0.83 \\ -2.10 & 2.80 \end{bmatrix} \, + \, z_3 \begin{bmatrix} -2.74 & 2.54 \\ -0.59 & 0.13 \end{bmatrix} \right) \vv_2 & = 0, \\
\left( \begin{bmatrix} 0.45 & 0.79 \\ -0.32 & 0.93 \end{bmatrix} \, + \, z_1 \, \begin{bmatrix} -0.98 & 1.18 \\ 3.59 & -1.27 \end{bmatrix} \!+\! z_2 \begin{bmatrix} 1.21 & -0.31 \\ -1.07 & 1.22 \end{bmatrix} \!+\! z_3 \begin{bmatrix} -2.09 & -2.34 \\ -0.69 & -1.37 \end{bmatrix} \right) \vv_3 &\!= 0.
\end{align*}
To keep the experiment reasonably efficient, we used only $20$ contour points in each real dimension of the boundary of the unit polydisc in $\C^3$, which has real dimension $5$. All the solutions to this problem that lie in the unit polydisc are real, and so they lie in $[-1,1]^3$. We plot the result in \cref{fig:tri}, keeping the solutions in the real unit box in $\mathbb{R}^3$. In \cref{tab:triBeyn}, we compare to multipareig~\cite{bor2025mpe}, which implements the classical operator determinants method efficiently.

Clearly a contour method is not a reasonable technique to solve a problem with only $8$ total solutions, and as expected our method is both much less accurate and much slower than multipareig. However, the comparison demonstrates that, while the number of contour points is insufficient to find the solutions accurately, we are able to locate all four solutions inside the region of interest. As our method works in principle, we believe that the techniques from \cref{sec:alg,sec:details} can be applied to optimize multiBeyn for more than two variables, which we expect will lead to the same increases in efficiency we have observed in our other experiments.

\begin{figure}
\begin{center}
\begin{overpic}[width=.8\textwidth]{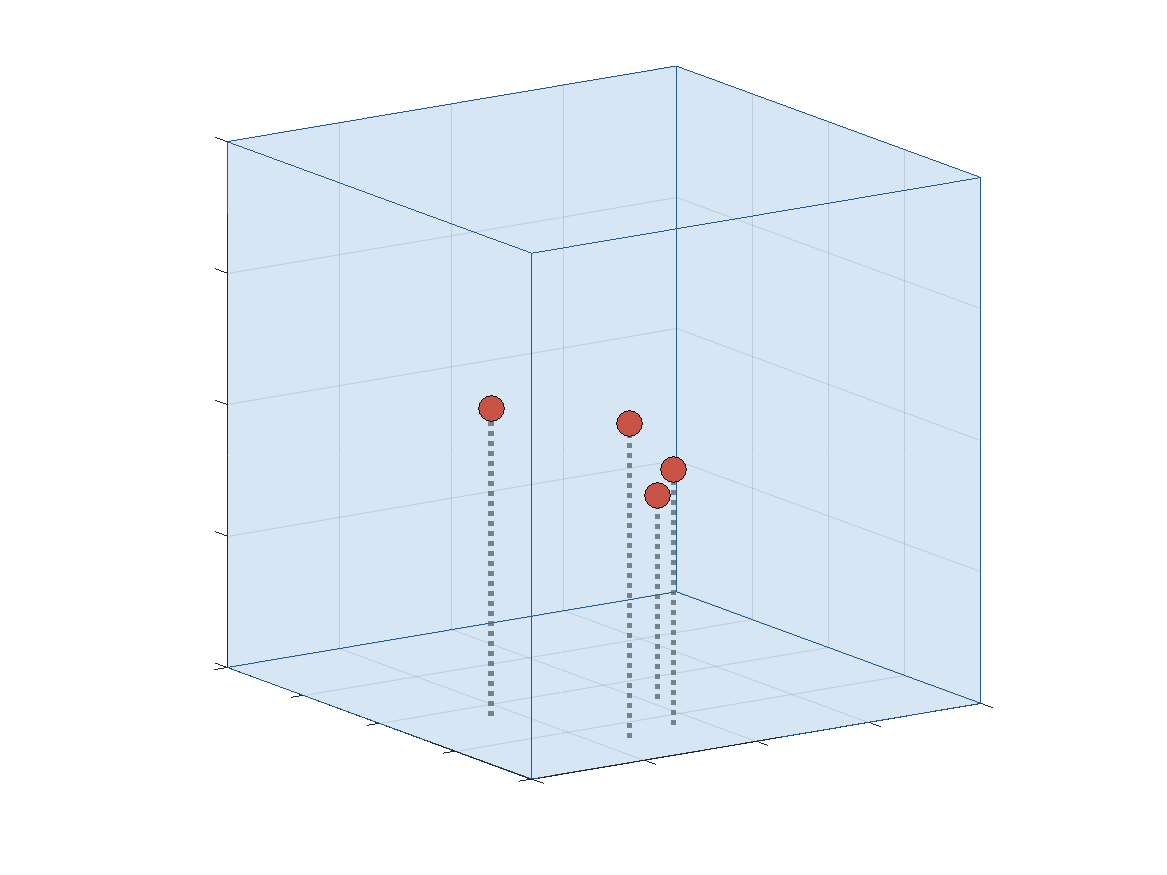}
\put(24,8){$z_1$}
\put(65,5){$z_2$}
            \put(12,40){\rotatebox{90}{$z_3$}}
             {\footnotesize
            \put(14.5,18.5){$-1$}
            \put(13.5,29){$-.5$}
            \put(16.5,40){$0$}
            \put(16,51.5){$.5$}
            \put(17,62){$1$}
            \put(19,15){$1$}
            \put(22,13){$.5$}
            \put(29,11){$0$}
            \put(32.5,8.5){$-.5$}
            \put(40,6.5){$-1$}
            \put(45,6){$-1$}
            \put(53,7){$-.5$}
            \put(65,8.5){$0$}
            \put(75,10){$.5$}
            \put(85,11.5){$1$}
            }
\end{overpic}
\end{center}
\caption{Solutions to a generic trivariate linear MEP in the real unit box $[-1,1]^3$. Our method triBeyn finds the four solutions in the unit box. The other $4$ solutions are outside the box.}
\label{fig:tri}
\end{figure}
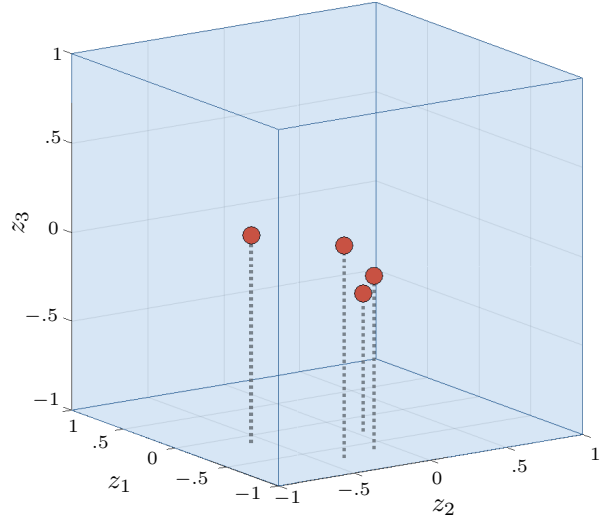

\begin{table}[h!]
\centering
\renewcommand{\arraystretch}{1.5}
\begin{tabular}{cc}
\hline
\textbf{Multipareig} & \textbf{triBeyn} \\ \hline
\textcolor{ForestGreen}{(0.0526 , -0.2715 , -0.2292 )} & \textcolor{ForestGreen}{(0.0529 , -0.2714 , -0.2296 )} \\ 
\textcolor{ForestGreen}{(-0.5986 , -0.1411 , 0.1692 )} & \textcolor{ForestGreen}{(-0.5981 , -0.1407 , 0.1695 )} \\ 
\textcolor{ForestGreen}{(-0.1176 , -0.6300 , -0.0300 )} & \textcolor{ForestGreen}{(-0.1165 , -0.6297 , -0.0321 )} \\ 
\textcolor{ForestGreen}{(-0.3602 , -0.6971 , 0.1972 )} & \textcolor{ForestGreen}{(-0.3700 , -0.6881 , 0.1863 )} \\
(9.1205 , 13.5920 , -22.1367 ) &  (outside region) \\ 
(0.8238 +1.1327i, 0.0845 -1.2101i, 0.5259 +0.0672i) & (outside region)  \\ 
(0.8238 -1.1327i, 0.0845 +1.2101i, 0.5259 -0.0672i) & (outside region)  \\ 
(0.0083 , 2.5704 , 15.7214 ) & (outside region) \\ \hline
\end{tabular}
\hspace{.3cm}
\caption{Comparison of solutions of multipareig~\cite{bor2025mpe} and triBeyn~\cite{graf2026beyn}. The green solutions all lie in the unit polydisc, and in the box $[-1,1]^3$. Our method computes all the solutions in the region of interest successfully.}
\label{tab:triBeyn}
\end{table}

\begin{appendices}

\end{appendices}

\bibliography{references}

\end{document}